\documentclass[amsfonts,12pt]{amsart}

\usepackage{amssymb}
\usepackage{graphicx}
\usepackage{stmaryrd}
\usepackage{color}
\usepackage{pifont}
\usepackage{enumerate}

\newcommand{\ee}{\varepsilon}

\newcommand{\N}{{\mathbb N}}
\newcommand{\Q}{{\mathbb Q}}
\newcommand{\R}{{\mathbb R}}

\newcommand{\cH}{{\mathcal H}}

\newcommand{\Sn}{S^{n-1}}

\newtheorem{thm}{Theorem}[section]

\newtheorem{lem}[thm]{Lemma}
\newtheorem{cor}[thm]{Corollary}

\newtheorem{prop}[thm]{Proposition}
\newtheorem{ex}[thm]{Example}

\newtheorem{prob}[thm]{Problem}

\newcommand{\cl}{{\mathrm{cl}}\,}
\newcommand{\conv}{{\mathrm{conv}}\,}
\newcommand{\Li}{{\mathrm{Li}}\,}
\newcommand{\Ls}{{\mathrm{Ls}}\,}

\newcommand{\aff}{{\mathrm{aff}}\,}
\newcommand{\lin}{{\mathrm{lin}}\,}

\newcommand{\di}{\diamondsuit}

\begin{document}
\hfill\today
\bigskip

\title{Convergence of symmetrization processes}
\author[Gabriele Bianchi, Richard J. Gardner, and Paolo Gronchi]
{Gabriele Bianchi, Richard J. Gardner, and Paolo Gronchi}
\address{Dipartimento di Matematica e Informatica ``U. Dini", Universit\`a di Firenze, Viale Morgagni 67/A, Firenze, Italy I-50134} \email{gabriele.bianchi@unifi.it}
\address{Department of Mathematics, Western Washington University,
Bellingham, WA 98225-9063} \email{richard.gardner@wwu.edu}
\address{Dipartimento di Matematica e Informatica ``U. Dini", Universit\`a di Firenze, Piazza Ghiberti 27, Firenze, Italy I-50122} \email{paolo.gronchi@unifi.it}
\thanks{First and third author supported in part by the Gruppo
Nazionale per l'Analisi Matematica, la Probabilit\`a e le loro
Applicazioni (GNAMPA) of the Istituto Nazionale di Alta Matematica (INdAM).  Second author supported in
part by U.S.~National Science Foundation Grant DMS-1402929.}

\maketitle

\begin{abstract}
Steiner and Schwarz symmetrizations, and their most important relatives, the Minkowski, Minkowski-Blaschke, fiber, inner rotational, and outer rotational symmetrizations, are investigated.  The focus is on the convergence of successive symmetrals with respect to a sequence of $i$-dimensional subspaces of $\R^n$.  Such a sequence is called universal for a family of sets if the successive symmetrals of any set in the family converge to a ball with center at the origin.  New universal sequences for the main symmetrizations, for all valid dimensions $i$ of the subspaces, are found, by combining two groups of results.  The first, published separately, provides finite sets ${\mathcal{F}}$ of subspaces such that reflection symmetry (or rotational symmetry) with respect to each subspace in ${\mathcal{F}}$ implies full rotational symmetry.  In the second, proved here, a theorem of Klain for Steiner symmetrization is extended to Schwarz, Minkowski, Minkowski-Blaschke, and fiber symmetrizations, showing that if a sequence of subspaces is drawn from a finite set ${\mathcal{F}}$ of subspaces, the successive symmetrals of any compact convex set converge to a compact convex set that is symmetric with respect to any subspace in ${\mathcal{F}}$ appearing infinitely often in the sequence.  It is also proved that for Steiner, Schwarz, and Minkowski symmetrizations, a sequence of $i$-dimensional subspaces is universal for the class of compact sets if and only if it is universal for the class of compact convex sets, and Klain's theorem is shown to hold for Schwarz symmetrization of compact sets.
\end{abstract}

\section{Introduction}

Around 1836, Jakob Steiner, in attempting to prove the isoperimetric inequality for convex bodies in $\R^n$, introduced the process that became known as Steiner symmetrization.  (Definitions of Steiner and other important symmetrizations, and some of their basic properties, can be found in Section~\ref{symm} below.)  Its most useful feature is that there are sequences of directions such that the corresponding successive Steiner symmetrals of a convex body always converge to a ball of the same volume.  Nowadays, this is employed in standard proofs of not only the isoperimetric inequality, but other potent geometric inequalities besides: the Blaschke-Santal\'o, Brunn-Minkowski, Busemann random simplex, and Petty projection inequalities, to name but a few.  Certain of these inequalities, such as the isoperimetric and Brunn-Minkowski inequalities, hold for more general sets, such as sets of finite perimeter. Many are affine invariant and therefore, at least in the convex geometry context, even more powerful; the Petty projection inequality, for example, is far stronger than the isoperimetric inequality. Several have more general, and sometimes even stronger, $L_p$ and Orlicz versions. For a sample, the reader can consult \cite{AB, FMP, HS, LYZ1, LYZ2, LZ, PP}.  Moreover, the impact of these inequalities extends far beyond geometry, since they often lead quickly to analytic versions; for example, the classical Sobolev inequality is equivalent to the isoperimetric inequality, and the Brunn-Minkowski and Petty projection inequalities yield, respectively, the Pr\'ekopa-Leindler inequality and Zhang's remarkable affine Sobolev inequality \cite{Z}.  This transition is explained at length in \cite{Gar02} and provides a portal to very wide applications to other areas such as probability and multivariate statistics, PDEs, and mathematical economics and finance.  For further information and references, see \cite[Chapter~9]{Gar06}, \cite[Chapter~9]{Gru07}, and \cite[Chapter~10]{Sch93}.

Nor is Steiner symmetrization the only symmetrization process by which valuable geometric and analytic inequalities can be established.  Minkowski and Schwarz symmetrization can be used to prove Urysohn-type inequalities and others; see, for example, \cite{FGM, GWW, PP}.  These and other symmetrization processes still are the main weapons of attack in the classic text of P\'{o}lya and Szeg\H{o} \cite{PS}, which stimulated further applications to PDEs, potential theory, and mathematical physics.  Here one finds inequalities for other set functions, such as the Poincar\'e-Faber-Szeg\H{o} inequality for capacity, the Faber-Krahn inequality for the first eigenvalue of the Laplacian, and P\'{o}lya's inequality for torsional rigidity.  The literature is vast and we can only point to \cite{B80, F00, Hay, Hen06, Kaw, K06, KP, LL, Str}, and the references given in these texts.

Our previous paper \cite{BGG} initiated a systematic study of symmetrizations in geometry, the basic notion being that of an {\em $i$-symmetrization}, a map $\di_H:\mathcal{E}\to \mathcal{E}_H$, where $H$ is a fixed $i$-dimensional subspace $H$ in $\R^n$, $i\in \{0,\dots,n-1\}$, $\mathcal{E}$ is a class of nonempty compact sets in $\R^n$, and $\mathcal{E}_H$ is the class of members of $\mathcal{E}$ that are $H$-symmetric (i.e., symmetric with respect to $H$).  Prototypical examples include Steiner symmetrization ($i=n-1$), Schwarz symmetrization ($i\in \{1,\dots,n-2\}$), and Minkowski symmetrization ($i\in \{0,\dots,n-1\}$). Both \cite{BGG} and the present paper focus on symmetrization of sets, but an analogous axiomatic framework for rearrangements of functions has been formulated by Brock and Solynin \cite{BS} and van Shaftingen \cite{VSPhD} (see also \cite{VSW}).  A different approach to the latter was proposed by the authors and Kiderlen \cite{BGGK}; see \cite[Appendix]{BGGK} for a comparison.

Most of the results in \cite{BGG} concern either classifications of Steiner and Minkowski symmetrizations in terms of their properties, or containment relations between symmetrals, akin to Lemma~\ref{1} below.  The focus of \cite[Section~8]{BGG}, however, is on the convergence of successive symmetrals, and the present paper grew from the seed planted there.  To discuss such results, we allow the subspace $H$ above to vary and call a collection $\di=\{\di_H : H\in{\mathcal{G}}(n,i)\}$ of maps an {\em $i$-symmetrization process}, where ${\mathcal{G}}(n,i)$ denotes the Grassmannian of $i$-dimensional subspaces in $\R^n$ (see Sections~\ref{subsec:notations} and~\ref{symm} for notation and terminology).  We borrow key notions from Coupier and Davydov \cite{CouD14}:  A sequence $(H_m)$ of $i$-dimensional subspaces is called {\em weakly $\di$-universal} if for any $k\in \N$, successive $\di$-symmetrals, with respect to $H_k, H_{k+1},\dots$, of any $E\in {\mathcal E}$ always converge to an origin-symmetric ball (depending on $E$ and $k$), and {\em $\di$-universal} if the ball is independent of $k$.  See Section~\ref{symm} for precise definitions. From \cite[Theorem~3.1]{CouD14} and earlier results, we know that when $i=n-1$ and ${\mathcal E}={\mathcal K}^n_n$, the class of convex bodies in $\R^n$, the four concepts (weakly) Steiner-universal and (weakly) Minkowski-universal are all equivalent.  In \cite[Theorem~8.1]{BGG}, it was shown that if $i\in \{1,\dots,n-1\}$ and $\di$ is an $i$-symmetrization process on ${\mathcal K}^n_n$ such that $I_H K\subset \di_HK\subset M_HK$ for $K\in {\mathcal K}^n_n$ and $H\in{\mathcal{G}}(n,i)$, where $I_HK$ and $M_HK$ are the inner rotational symmetral and Minkowski symmetral of $K$, respectively, then any Minkowski-universal sequence in ${\mathcal{G}}(n,i)$ is also weakly $\di$-universal.  This allowed the same conclusion to be drawn, in \cite[Corollary~8.2]{BGG}, for any $i$-symmetrization process satisfying Properties 1, 4, and 7 in Section~\ref{symm}.

Steiner-universal sequences are far from rare. In \cite[Proposition~3.3]{CouD14}, a result from \cite{BKLYZ} is used to show that they can be drawn from any set of $(n-1)$-dimensional subspaces whose unit normal vectors are dense in $S^{n-1}$.  Steiner-universal sequences affording a particularly fast rate of convergence are provided by Klartag \cite{Kla04}; the construction depends on that in \cite{Kla02} and involves concatenating finite sequences of symmetrizations, some with respect to a deterministic choice of directions and some with respect to a random choice.  Random Steiner symmetrizations that produce Steiner-universal sequences almost surely have been described by Burchard and Fortier \cite[Corollary~2.3]{BF} and Coupier and Davidov \cite[Proposition~3.2]{CouD14}.  In the former, similar results are obtained first for random polarizations and then combined with van Shaftingen's universal approximation of Steiner symmetrization by polarizations in \cite{VS1}, \cite{VS2}.  We note in passing that van Schaftingen's approximations can be adapted to other symmetrizations, including some that play no major role in the present paper, such as the spherical cap symmetrization; see \cite[Section~4.3]{VS2}.

It is important to bear in mind that in general the limit of a sequence of successive Steiner symmetrals of a convex body may not exist; see \cite[Example~2.1]{BBGV}.  However, a remarkable theorem of Klain \cite[Theorem~5.1]{Kla12} reveals another source of Steiner-universal sequences.  This states that if a sequence $(H_m)$ of $(n-1)$-dimensional subspaces is drawn from a finite set ${\mathcal{F}}$ of such subspaces, the successive Steiner symmetrals of any compact convex set, with respect to the subspaces in $(H_m)$, will converge to a compact convex set that is symmetric with respect to each subspace in ${\mathcal{F}}$ that occurs infinitely often in $(H_m)$. From this, Klain \cite[Corollary~5.4]{Kla12} is able to conclude that if the unit normal vectors to the subspaces in ${\mathcal{F}}$ contains an irrational basis and each element of the basis occurs infinitely often as a unit normal vector to some $H_m$, then $(H_m)$ is Steiner-universal.

A major goal of this paper is to obtain more information about universal sequences of subspaces.  In Theorem~\ref{4}, we adapt an argument from \cite{BBGV} to extend Klain's theorem to fiber symmetrization for $i\in \{1,\dots,n-1\}$, defined by (\ref{fhjk}) below.  This little-known but fundamental symmetrization reduces to Steiner symmetrization when $i=n-1$.  Fiber symmetrization becomes Minkowski symmetrization only when $i=0$, but nevertheless we prove in Theorem~\ref{corthm4} that Klain's theorem also holds for Minkowski symmetrization when $i\in \{1,\dots,n-1\}$, and indeed for any symmetrization with Properties 1, 4, 7, and 9 from Section~\ref{symm}.  In Theorem~\ref{thm5cor}, we prove a version of Klain's theorem for Schwarz symmetrization; in this case the limit compact convex set is rotationally symmetric with respect to each subspace in ${\mathcal{F}}$ that occurs infinitely often in $(H_m)$.  We also show that the same result holds for Minkowski-Blaschke symmetrization.  Since Klain's theorem does not apply when $i=0$, our results extend the theorem to all the main subspace symmetrizations in the literature except Blaschke symmetrization.  The latter is somewhat of an oddball, lacking even the monotonicity property (see \cite[Theorem~3.1]{BGG}), and we pay no attention to it in this paper beyond stating Problem~\ref{probmar28}.

Using our extensions of Klain's theorem, we are able to exhibit new universal sequences when ${\mathcal{E}}={\mathcal K}^n_n$.  Theorem~\ref{21June} does this for Minkowski symmetrization, for $i\in \{1,\dots,n-1\}$ and Theorem~\ref{22} serves the same purpose for Schwarz and Minkowski-Blaschke symmetrization.  Moreover, Theorem~\ref{21June} and \cite[Corollary~8.2]{BGG} yield, for all $i\in \{1,\dots,n-1\}$, weakly $\di$-universal sequences for any $i$-symmetrization process $\di$ satisfying Properties 1, 4, and 7 in Section~\ref{symm}. As far as we know, no explicit universal sequences have appeared before in the literature for $i\in \{1,\dots,n-2\}$, except for isolated results such as that of Tonelli \cite{Ton15} for Schwarz symmetrization when $i=1$ and $n=3$.

The transition from extensions of Klain's theorem to the existence of universal sequences requires an extra step.  From Klain's theorem we know that the limit compact convex set is symmetric with respect to each subspace in the finite set ${\mathcal{F}}$.  Therefore if ${\mathcal{F}}$ can be chosen so that reflection symmetry (or, for Schwarz and Minkowski-Blaschke symmetrization, rotational symmetry) in each subspace in ${\mathcal{F}}$ implies full rotational symmetry, then the limit compact convex set must be an origin-symmetric ball.  Hence, in the case of reflection symmetry, the problem is to find a finite set $\mathcal{F}$ such the closure of the subgroup of $O(n)$ generated by the reflections in the subspaces in $\mathcal{F}$ acts transitively on $S^{n-1}$.  Here we apply results from our paper \cite{BGG3}, summarized in Section~\ref{rotationalsymmetry}, that build on earlier work of
Burchard, Chambers, and Dranovski \cite{BCD} (see also \cite{EP}).

In Section~\ref{compact}, we examine symmetrization of compact sets. Steiner, Schwarz, and Minkowski symmetrization make sense for compact sets, and some convergence results are available in this setting; see, for example, \cite{BBGV}, \cite{Ton15}, \cite{VS3}, \cite{Vol}, and \cite{V}.  To these we add the observation, in Theorem~\ref{july24thm}, that Klain's theorem holds for the Schwarz symmetrization of compact sets.  By \cite[Examples~2.1 and~2.4]{BBGV}, with the sequence ($\alpha_m$) of reals there chosen so that their sum converges, we know that the limit of a sequence of successive Steiner symmetrals of a compact set may exist but be non-convex.  Thus there is no obvious direct relationship between convergence for convex bodies and convergence for compact sets. Nevertheless, in Theorems~\ref{thmdec17} and~\ref{thmjan9}, we prove that if $i\in \{1,\dots,n-1\}$, any Schwarz-universal (or Minkowski-universal) sequence for the class of convex bodies is also Schwarz-universal (or Minkowski-universal, respectively) for the class of compact sets.  Here we regard Schwarz symmetrization for $i=n-1$ as Steiner symmetrization.  These results, together with those mentioned above, show that the eight properties (weakly) Steiner-universal for convex bodies, (weakly) Minkowski-universal for convex bodies, (weakly) Steiner-universal for compact sets, and (weakly) Minkowski-universal for compact sets are all equivalent.  A tool in the proofs of Theorems~\ref{thmdec17} and~\ref{thmjan9} is the notion of the Kuratowski limit superior or inferior of a sequence of sets.

In addition to the principal results discussed above, we prove a new containment result relating various symmetrals, Lemma~\ref{1}; an analog of \cite[Theorem~8.1]{BGG} for Schwarz symmetrization, Theorem~\ref{2}; and a characterization of Schwarz symmetrization in terms of its properties, Theorem~\ref{Schwarz}.

Throughout the paper we provide, whenever we can, examples to show that the various assumptions we make are necessary.  The final Section~\ref{problems} lists some problems left open by our study.

For most applications, for example in establishing geometric inequalities, basic convergence results for symmetrizations suffice.  We end this introduction, however, with a nod to a quite different line of investigation, namely, determining the least number of successive symmetrals required to transform a set into one in some sense close to an origin-symmetric ball.  Such results on rates of convergence often require very delicate analysis, as evidenced by the deep work of Bourgain, Klartag, Lindenstrauss, Milman, and others, and have obvious applications to geometric asymptotic analysis, for example.  See \cite{Kla02}, \cite{Kla04}, and the references given there. A worthwhile endeavor, but one that must wait for future research, would be to determine rates of convergence for symmetrizations in the general setting considered here.

We are very grateful to a referee whose extensive knowledge and incisive remarks led to substantial improvements.

\section{Preliminaries}\label{subsec:notations}

As usual, $S^{n-1}$ denotes the unit sphere and $o$ the origin in Euclidean $n$-space $\R^n$ with Euclidean norm $\|\cdot\|$.  We assume throughout that $n\ge 2$.   The term {\em ball} in $\R^n$ will always mean an $n$-dimensional ball unless otherwise stated. The unit ball in $\R^n$ will be denoted by $B^n$ and $B(x,r)$ is the ball with center $x$ and radius $r$. If $x,y\in \R^n$, we write $x\cdot y$ for the inner product and $[x,y]$ for the line segment with endpoints $x$ and $y$. If $x\in \R^n\setminus\{o\}$, then $x^{\perp}$ is the $(n-1)$-dimensional subspace orthogonal to $x$. Throughout the paper, the term {\em subspace} means a linear subspace.

If $X$ is a set,  we denote by $\lin X$, $\conv X$, $\cl X$, and $\dim X$ the {\it linear hull}, {\it convex hull}, {\it closure}, and {\it dimension} (that is, the dimension of the affine hull) of $X$, respectively.  If $H$ is a subspace of $\R^n$, then $X|H$ is the (orthogonal) projection of $X$ on $H$ and $x|H$ is the projection of a vector $x\in \R^n$ on $H$.

If $X$ and $Y$ are sets in $\R^n$ and $t\ge 0$, then $tX=\{tx:x\in X\}$ is the dilate of $X$ by the factor $t$ and
$$X+Y=\{x+y: x\in X, y\in Y\}$$
denotes the {\em Minkowski sum} of $X$ and $Y$.

When $H$ is a fixed subspace of $\R^n$, we write $R_HX$ for the {\em reflection} of $X$ in $H$, i.e., the image of $X$ under the map that takes $x\in \R^n$ to $2(x|H)-x$.  If $R_HX=X$, we say $X$ is {\em $H$-symmetric}. If $H=\{o\}$, we instead write $-X=(-1)X$ for the reflection of $X$ in the origin and {\it $o$-symmetric} for $\{o\}$-symmetric.  A set $X$ is called {\em rotationally symmetric} with respect to the $i$-dimensional subspace $H$ if for all $x\in H$, $X\cap (H^{\perp}+x)$ is a union of $(n-i-1)$-dimensional spheres, each with center at $x$.  These are just the sets that are invariant under each element of $O(n)$ or $SO(n)$ that fixes $H$, i.e., act as the identity on $H$.  If $\dim H=n-1$, then $X$ is rotationally symmetric with respect to $H$ if and only if it is $H$-symmetric. If $K$ is a compact convex set, then $K$ is rotationally symmetric with respect to $H$ precisely when for all $x\in H$, $K\cap (H^{\perp}+x)=r_x(B^n\cap H^{\perp})+x$ for some $r_x\ge 0$.  The term {\em $H$-symmetric spherical cylinder} will always mean a set of the form $(B(x,r)\cap H)+s(B^n\cap H^{\perp})=(B(x,r)\cap H)\times s(B^n\cap H^{\perp})$, where $r,s>0$.  Thus these sets are solid cylinders whose $H$- and $H^{\perp}$-components are centered at $x$ and $o$, respectively.  Of course, $H$-symmetric spherical cylinders are rotationally symmetric with respect to both $H$ and $H^{\perp}$.

The phrase {\em translate orthogonal to $H$} means translate by a vector in $H^{\perp}$.

We write ${\mathcal H}^k$ for $k$-dimensional Hausdorff measure in $\R^n$, where $k\in\{1,\dots, n\}$.

The Grassmannian of $k$-dimensional subspaces in $\R^n$ is denoted by ${\mathcal{G}}(n,k)$.

We denote by ${\mathcal C}^n$ the class of nonempty compact subsets of $\R^n$. Let ${\mathcal K}^n$ be the class of nonempty compact convex
subsets of $\R^n$ and let ${\mathcal K}^n_n$ be the class of {\em convex bodies}, i.e., members of ${\mathcal K}^n$ with interior points.  A subscript $s$ denotes the $o$-symmetric sets in these classes.  If $K\in {\mathcal K}^n$, then
$$
h_K(x)=\sup\{x\cdot y: y\in K\},
$$
for $x\in\R^n$, defines the {\it support function} $h_K$ of $K$.  The texts by Gruber \cite{Gru07} and Schneider \cite{Sch93} contain a wealth of useful information about convex sets and related concepts such as the {\em intrinsic volumes} $V_j$, $j\in\{1,\dots, n\}$ (see \cite[Section~6.4]{Gru07}, \cite[p.~208]{Sch93}, and also \cite[Appendix~A]{Gar06}).  In particular, if $K\in {\mathcal K}^n$, then $V_{n}(K)={\mathcal H}^n(K)$ is the {\em volume} of $K$, and $V_{n-1}(K)$ and $V_1(K)$ are (up to constants independent of $K$) the {\em surface area} and {\em mean width} of $K$, respectively; see \cite[p.~104]{Gru07} or \cite[(A.28), (A.35), (A.50)]{Gar06}. Intrinsic volumes are independent of the dimension of the ambient space, so that if $\dim K=k$, then $V_k(K)={\mathcal H}^k(K)$, and in this case we prefer to write $V_k(K)$.  By $\kappa_n$ we denote the volume $V_n(B^n)$ of the unit ball in $\R^n$.

\section{Rotational symmetry via symmetries in finitely many subspaces}\label{rotationalsymmetry}

In this section, we summarize the results we need from \cite{BGG3} that address the problem of finding finite sets of $i$-dimensional subspaces such that reflections in these subspaces, or (full) rotational symmetries with respect to subspaces, generate full rotational symmetry. We refer to \cite{BGG3} for comments on related earlier work of Burchard, Chambers, and Dranovski \cite{BCD} and others.

\begin{thm}\label{teo:sphericalsym_lines}
{\em (\cite[Theorem~3.2]{BGG3}.)} Let $H_j\in {\mathcal{G}}(n,1)$, $j=1,\dots,n$, be such that

\noindent{\rm{(i)}} at least two of them form an angle that is an irrational multiple of $\pi$,

\noindent{\rm{(ii)}} $H_1+\cdots+H_n=\R^n$, and

\noindent{\rm{(iii)}} $\{ H_1,\dots,H_n\}$ cannot be partitioned into two mutually orthogonal nonempty subsets.

If $E\subset\Sn$ is nonempty, closed, and such that $R_{H_j}E=E$, $j=1,\dots,n$, then $E=\Sn$.
\end{thm}

Let $2\leq i\leq n/2$ and let $\alpha_1,\dots,\alpha_i$ be an increasing sequence in $(0,\pi/2)$ such that $\pi, \alpha_1,\dots, \alpha_i$ are linearly independent over $\Q$. Let $H_1, H_2, H_3\in{\mathcal{G}}(n,i)$ be defined by
\begin{equation}\label{choice_of_the_planes}
\begin{aligned}
H_1=&\lin\{ e_1,e_3,\dots,e_{2i-1}\},\\
H_2=&\lin\{\cos \alpha_j\, e_{2j-1}+\sin \alpha_j\, e_{2j} : j=1,\dots,i\},\quad{\text{and}}\\
H_3=&\lin\{\{\cos \alpha_1\, e_{1}+\sin \alpha_1\, e_{2i}\}\cup\{\cos \alpha_j\, e_{2j-1}+\sin \alpha_j\, e_{2j-2}\} : j=2,\dots,i\}.
\end{aligned}
\end{equation}

\begin{thm}\label{finthm}
{\em (\cite[Theorem~3.8]{BGG3}.)} Let $k\ge 3$ and let $H_j\in{\mathcal{G}}(n,i)$, $j=1,\dots,k$. If $2\le i\le n/2$, assume that

{\noindent{\rm{(i)}}} $H_1$, $H_2$, and $H_3$ are as in \eqref{choice_of_the_planes},

{\noindent{\rm{(ii)}}}  $H_1+\dots+H_k=\R^n$, and

{\noindent{\rm{(iii)}}} for each $j=3,\dots,k-1$,
$$
H_{j+1}\cap\left(H_1+\dots+H_{j}\right)^\perp=\{o\}.
$$
If $n/2< i\leq n-2$, assume that {\em (i)--(iii)} are satisfied with each $H_j$ replaced by $H_j^\perp$.  If $E\subset \Sn$ is nonempty, closed, and such that $R_{H_j} E=E$ for $j=1,\dots,k$, then $E=\Sn$.
\end{thm}

\begin{cor}\label{cor21June}
{\em (\cite[Corollary~3.9]{BGG3}.)} Let $1\leq i\leq n-1$ and let
$$
k=
\begin{cases}
n, & \text{if $i=1$ or $i=n-1$,}\\
\lceil n/i\rceil+1, & \text{if $1<i\leq n/2$,}\\
\lceil n/(n-i)\rceil+1, & \text{if $n/2\leq i<n-1$.}
\end{cases}
$$
There exist $H_j\in{\mathcal{G}}(n,i)$, $j=1,\dots,k$, such that if $E\subset \Sn$ is nonempty, closed, and such that $R_{H_j} E=E$ for $j=1,\dots, k$, then $E=\Sn$.  Hence, if $F\subset \R^n$ is closed and $R_{H_j} F=F$ for $j=1,\dots, k$ (or $K\in{\mathcal{K}}^n_n$ and $R_{H_j}K=K$ for $j=1,\dots, k$), then $F$ is a union of $o$-symmetric spheres (or $K$ is an $o$-symmetric ball, respectively).
\end{cor}

\begin{thm}\label{new}
{\em (\cite[Theorem~4.1]{BGG3}.)} Let $H_1,\dots,H_k$ be subspaces in $\R^n$ such that $1\le \dim H_j\le n-2$ for $j=1,\dots,k$.  The following statements are equivalent.

\noindent{\rm{(i)}} $H_1^{\perp}+\cdots+H_k^{\perp}=\R^n$ and $\{H_1^{\perp},\dots,H_k^{\perp}\}$ cannot be partitioned into two mutually orthogonal nonempty subsets.

\noindent{\rm{(ii)}} If $E\subset S^{n-1}$ is nonempty and closed, and for each $j=1,\dots,k$ and $x\in E$, we have $\Sn\cap (H_j^{\perp}+x)\subset E$, then $E=S^{n-1}$.

\noindent{\rm{(iii)}}
If $F\subset \R^n$ is closed and invariant under any rotation that fixes $H_j$ for each $j=1,\dots,k$, then $F$ is a union of $o$-symmetric spheres.

\noindent{\rm{(iv)}}
If $K\in{\mathcal{K}}^n_n$ is rotationally symmetric with respect to $H_j$ for each $j=1,\dots,k$, then $K$ is an $o$-symmetric ball.
\end{thm}

\begin{cor}\label{cor22June}
{\em (\cite[Corollary~4.2]{BGG3}.)} Let $1\leq i\leq n-2$ and let $k=\lceil n/(n-i)\rceil$. There exist $H_j\in{\mathcal{G}}(n,i)$, $j=1,\dots,k$, such that the statements in Theorem~\ref{new} hold.
\end{cor}

\section{$i$-Symmetrization and $i$-symmetrization processes}\label{symm}

Let $i\in\{0,\dots,n-1\}$ and let $H\in {\mathcal{G}}(n,i)$ be fixed. Let $\mathcal{B}\subset {\mathcal{C}}^n$ be a class of nonempty compact sets in $\R^n$ and let $\mathcal{B}_H$ denote the subclass of members of $\mathcal{B}$ that are $H$-symmetric. We call a map ${\di_H}:{\mathcal{B}}\rightarrow{\mathcal{B}}_H$ an {\em $i$-symmetrization} on $\mathcal{B}$ (with respect to $H$).  If $K\in {\mathcal{B}}$, the corresponding set ${\di_H}K$ is called a {\em symmetral}.  We consider the following properties, where it is assumed that the class ${\mathcal{B}}$ is appropriate for the properties concerned and that they hold for all $K,L\in {\mathcal{B}}$.  Recall that $R_HK$ is the reflection of $K$ in $H$.

\smallskip

1. ({\em Monotonicity} or {\em strict monotonicity}) \quad $K\subset L \Rightarrow \di_H K\subset \di_H L$ (or $\di_H K\subset \di_H L$ and $K\neq L \Rightarrow \di_H K\neq \di_H L$, respectively).

2. ({\em $f$-preserving}) \quad $f(\di_H K)=f(K)$, where $f:{\mathcal{B}}\to [0,\infty)$ is a set function.  In particular, we can take $f=V_j$, $j=1,\dots,n$, the $j$th intrinsic volume, though we generally prefer to write {\em mean width preserving}, {\em surface area preserving}, and {\em volume preserving}  when $j=1$, $n-1$, and $n$, respectively.

3. ({\em Idempotence}) \quad $\di_H^2 K=\di_H(\di_H K)=\di_H K$.

4. ({\em Invariance on $H$-symmetric sets})\quad $R_HK=K\Rightarrow\di_H K=K$.

5. ({\em Invariance on $H$-symmetric spherical cylinders})\quad If $K=D_r(x)+s(B^n\cap H^{\perp})$, where $s>0$ and $D_r(x)\subset H$ is the $i$-dimensional ball with center $x$ and radius $r>0$, then $\di_H K=K$.

6. ({\em Projection invariance}) \quad $(\di_H K)|H=K|H$.

7. ({\em Invariance under translations orthogonal to $H$ of $H$-symmetric sets}) \quad If $R_HK=K$ and $y\in H^{\perp}$, then $\di_H(K+y)=\di_H K$.

8. ({\em Rotational symmetry}) \quad  $\di_HK$ is rotationally symmetric with respect to $H$.

9. ({\em Continuity}) \quad $\di_HK_m\to \di_HK$ as $m\to\infty$ whenever $K_m\to K$ as $m\to\infty$ in the Hausdorff metric (cf.~\cite[p.~2310]{GHW}).

The six main symmetrizations of interest for this paper are defined and their properties summarized below.  In each, we take $K\in {\mathcal{K}}^n$ and $H\in {\mathcal{G}}(n,i)$; additional information and references can be found in \cite{BGG}.

Let $i\in \{1,\dots,n-1\}$.  Let $S_HK$ be such that for each $(n-i)$-dimensional plane $G$ orthogonal to $H$ and meeting $K$, the set $G\cap S_HK$ is a (possibly degenerate) $(n-i)$-dimensional closed ball with center in $H$ and $(n-i)$-dimensional volume equal to that of $G\cap K$.  When $i=n-1$, $G\cap S_HK$ is a (possibly degenerate) closed line segment with midpoint in $H$ and length equal to that of $G\cap K$, and $S_HK$ is called the {\em Steiner symmetral} of $K$ with respect to $H$.  When $i\in \{1,\dots,n-2\}$, $S_HK$ is the {\em Schwarz symmetral} of $K$ with respect to $H$.  See \cite[p.~57]{BGG}.  Steiner and Schwarz symmetrization are the original classical ones with myriad applications; see the references given in \cite[p.~52]{BGG}.

Let $i\in \{0,\dots,n-1\}$. The {\em Minkowski symmetral} of $K$ with respect to $H$ is defined by
\begin{equation}\label{april31}
M_HK=\frac12 K+\frac12 R_HK.
\end{equation}
(Recall that $R_HK$ is the reflection of $K$ in $H$.) See \cite[p.~57]{BGG}.

Let $i\in \{1,\dots,n-2\}$. The {\em Minkowski-Blaschke symmetral} $\overline{M}_HK$ of $K$ is defined for $u\in S^{n-1}$ via
$$
h_{\overline{M}_HK}(u)=
\begin{cases}
\frac{1}{{{\cH}^{n-i-1}}(S^{n-1}\cap (H^{\perp}+u))}
\int_{S^{n-1}\cap (H^{\perp}+u)}h_K(v)\,dv, & \text{if ${{\cH}^{n-i-1}}(S^{n-1}\cap (H^{\perp}+u))\neq 0$,}\\
h_K(u), & \text{otherwise.}
\end{cases}
$$
Thus the support function $h_K(u)$ of $K\in {\mathcal{K}}^n$ at $u\in S^{n-1}$ is replaced by the average of $h_K$ over the subsphere of $S^{n-1}$ orthogonal to $H$ and containing $u$.  See \cite[p.~58]{BGG} and note that we have simplified and corrected the definition given there.   We can extend the definition to $i=n-1$ if we interpret it to mean that $\overline{M}_HK=M_HK$ in this case.

Let $i\in \{0,\dots,n-1\}$.  The {\em fiber symmetral} $F_{H}K$ of $K$ is defined by
\begin{eqnarray}\label{fhjk}
F_{H}K&=&\bigcup_{x\in H}\left(\frac12 (K\cap (H^{\perp}+x))+\frac12R_H(K\cap (H^{\perp}+x))\right)\nonumber\\
&=&\bigcup_{x\in H}\left(\frac12 (K\cap (H^{\perp}+x))+\frac12(R_HK\cap (H^{\perp}+x))\right).
\end{eqnarray}
See \cite[p.~58]{BGG}.  Thus each non-degenerate section of $F_HK$ by an $(n-i)$-dimensional subspace orthogonal to $H$ is the Minkowski symmetral of the corresponding section of $K$.

Let $i\in\{1,\dots,n-1\}$. The {\em inner rotational symmetral} $I_HK$ of $K$ is such that for each $(n-i)$-dimensional plane $G$ orthogonal to $H$ and meeting $K$, the set $G\cap I_HK$ is a (possibly degenerate) $(n-i)$-dimensional ball with center in $H$ and radius equal to that of the (possibly degenerate) largest $(n-i)$-dimensional ball contained in $G\cap K$.  The {\em outer rotational symmetral} $O_HK$ of $K$ is the intersection of all rotationally symmetric with respect to $H$ convex bodies for which some translate orthogonal to $H$ contains $K$.

The Minkowski, Minkowski-Blaschke, fiber, inner rotational, and outer rotational symmetrizations are all important and very natural relatives of Steiner and Schwarz symmetrizations.  Indeed, there is a remarkable duality \cite[Theorem~7.1]{BGG} between Steiner and Minkowski symmetrization when $i=n-1$, which extends to one between fiber and Minkowski symmetrization when $i\in\{0,\dots,n-1\}$:  The fiber (Steiner, if $i = n-1$) symmetral of $K\in {\mathcal{K}}^n$ is the {\em union} of all $H$-symmetric compact convex sets such that some translate orthogonal to $H$ {\em is contained in} $K$, while the Minkowski symmetral of $K$ is the {\em intersection} of all $H$-symmetric compact convex sets such that some translate orthogonal to $H$ {\em contains} $K$.  The inner and outer rotational symmetrizations, introduced in \cite[Section~5]{BGG}, display exactly the same duality in a rotationally symmetric setting.  Minkowski-Blaschke symmetrization extends to $i\in\{1,\dots,n-2\}$ the case $i=n-1$ of Minkowski symmetrization in a different, but nevertheless direct, way to using (\ref{april31}).  Despite their names, both Minkowski and Minkowski-Blaschke symmetrizations were introduced by Blaschke and have found various uses (see \cite[Sections~1 and~3]{BGG}).

Some of the applications of the symmetrizations discussed in the previous paragraph stem from containment relations, to which we now turn our attention.  We begin with the following new result, which reduces to the known fact that $S_HK\subset M_HK$ when $i=n-1$.

\begin{lem}\label{1}
If $H\in {\mathcal{G}}(n,i)$, $i\in\{1,\dots,n-1\}$, and $K\in {\mathcal{K}}^n$, then $S_HK\subset \overline{M}_HK\subset O_HK$.
\end{lem}

\begin{proof}
For the right-hand inclusion, let $L$ be rotationally symmetric with respect to $H$ and such that $K\subset L+y$, where $y\in H^{\perp}$.  Since $\overline{M}_H$ is monotonic and invariant under translations of $H$-symmetric sets orthogonal to $H$, we obtain
$$\overline{M}_HK\subset \overline{M}_H(L+y)=\overline{M}_HL=L,$$
where the previous equality follows directly from the definition of $\overline{M}_H$ and the fact that $L$ is rotationally symmetric with respect to $H$. Therefore $\overline{M}_HK\subset L$ and hence $\overline{M}_HK\subset O_HK$ by the definition of $O_H$.

For the left-hand inclusion, let $x\in H$.  By the definition of $S_HK$, we have
$$(S_HK)\cap (H^{\perp}+x)=s_x(B^n\cap H^{\perp})+x,$$
for some $s_x\ge 0$ satisfying
\begin{equation}\label{jan221}
V_{n-1}(K\cap (H^{\perp}+x))=V_{n-1}(s_x(B^n\cap H^{\perp})+x).
\end{equation}
By the definition of $\overline{M}_HK$, we have that for $u\in S^{n-1}$, $h_{\overline{M}_HK}(u)$ is the average of $h_K(v)$ over $v\in S^{n-1}\cap (H^{\perp}+u)$.  It follows that
$$\overline{M}_H(K\cap (H^{\perp}+x))=t_x(B^n\cap H^{\perp})+x\subset H^{\perp}+x,$$
for some $t_x\ge 0$ satisfying
\begin{equation}\label{jan222}
V_{1}(K\cap (H^{\perp}+x))=V_{1}(t_x(B^n\cap H^{\perp})+x).
\end{equation}
By Urysohn's inequality (\cite[Theorem~9.3.2, p.~365]{Gar06} with $p=1$ or \cite[(7.21), p.~382]{Sch93}) in $H^{\perp}+x\approx \R^{n-1}$, the inequalities (\ref{jan221}) and (\ref{jan222}) imply that $s_x\le t_x$.  It follows, by the monotonicity of $\overline{M}_H$, that
$$(S_HK)\cap (H^{\perp}+x)\subset \overline{M}_H\left( K\cap (H^{\perp}+x)\right) \subset (\overline{M}_HK)\cap (H^{\perp}+x)$$
and hence that $S_HK\subset \overline{M}_HK$.
\end{proof}

The following theorem summarizes all the known inclusions between the various symmetrals.  It should be added that $F_HK=M_HK$ when $i=0$.

\begin{thm}\label{thmapril3}
If $H\in {\mathcal{G}}(n,i)$, $i\in \{1,\dots,n-1\}$, and $K\in {\mathcal{K}}^n$, then
\begin{equation}\label{april32}
I_HK\subset F_HK\subset M_HK\subset O_HK
\end{equation}
and
\begin{equation}\label{april33}
I_HK\subset S_HK\subset \overline{M}_HK\subset O_HK.
\end{equation}
When $i=n-1$, we have $I_HK=F_HK=S_HK$ and $M_HK=\overline{M}_HK=O_HK$.
\end{thm}

\begin{proof}
The left-hand inclusion in (\ref{april32}) follows easily from the definitions of $I_HK$ and $F_HK$.  The middle inclusion is a consequence of \cite[Corollary~7.3]{BGG} and the right-hand inclusion is noted in \cite[p.~66]{BGG}.  The left-hand inclusion in (\ref{april33}) is also noted in \cite[p.~66]{BGG} and the remaining inclusions are given by Lemma~\ref{1}.  The stated equalities all follow directly from the definitions.
\end{proof}

The seven symmetrizations defined in this section possess all the Properties~1 and 3--8 above, except that only Steiner, Minkowski, and fiber symmetrization are invariant on $H$-symmetric sets and neither Minkowski nor fiber symmetrization is rotationally symmetric when $i<n-1$.  Regarding Property~2, $S_H$ (and hence $F_H$ when $i=n-1$) preserves volume, and $M_H$ (and hence $F_H$ when $i=0$) and $\overline{M}_H$ preserve mean width.  For more details, see \cite{BGG}, especially Sections~3 and~5.  Regarding Property~9, which was not considered in \cite{BGG}, all seven symmetrizations are continuous on ${\mathcal{K}}^n_n$ and all except $S_H$, $F_H$ (for $i\in\{1,\dots,n-1\}$), and $I_H$ are continuous on ${\mathcal{K}}^n$.  (For $O_H$, let $\ee>0$ and let $K, L\in {\mathcal{K}}^n$ satisfy $K \subset L + \ee B^n$.  Note that $O_H L$ is the smallest rotationally symmetric with respect to $H$ compact convex set with a translate containing $L$, because the intersection of rotationally symmetric convex bodies is a rotationally symmetric compact convex set.  Then
$L \subset O_H L + y$ for some $y \in H^{\perp}$ and hence
$K \subset O_H L + \ee B^n + y$.  Therefore $O_H L + \ee B^n$ is a rotationally symmetric convex body with a translate containing $K$ and it follows that $O_H K \subset O_H L + \ee B^n$. The continuity of $O_H$ on ${\mathcal{K}}^n$ follows.  The non-continuity of $S_H$, $F_H$ (for $i\in\{1,\dots,n-1\}$), and $I_H$ on ${\mathcal{K}}^n$ can be verified by considering a sequence of origin-symmetric $(n-i)$-dimensional balls contained in $(n-i)$-dimensional subspaces whose intersection with $H^\perp$ has dimension less than $n-i$ and which converge to an $(n-i)$-dimensional ball contained in $H^\perp$.)

Table~\ref{newtable} summarizes this information.

\begin{table}
\begin{center}
\begin{tabular}{|c|c l|c|c|c|c|c|c|c|c|c|c|} \hline
 & & &\rotatebox[origin=c]{90}{Monotonic} & \rotatebox[origin=c]{90}{$V_j$-preserving} & \rotatebox[origin=c]{90}{Idempotent} & \rotatebox[origin=c]{90} {Inv. $H$-sym. sets} & \rotatebox[origin=c]{90}{Inv.~$H$-sym.~sph.~cyl.} & \rotatebox[origin=c]{90}{Projection inv.} & \rotatebox[origin=c]{90}{\ Inv. translations\ } & \rotatebox[origin=c]{90}{Rotational symm.} & \rotatebox[origin=c]{90}{Continuous on ${\mathcal{K}}^n_n$}
 & \rotatebox[origin=c]{90}{Continuous on ${\mathcal{K}}^n$}\\
Name & Symbol & &1&2&3&4&5&6&7&8&9&9 \\ \hline
Steiner &$S_H$, & $i=n-1$ & s\checkmark & $V_n$ & \checkmark & \checkmark & \checkmark & \checkmark & \checkmark & \checkmark & \checkmark& \small{\ding{53}}\\ \hline
Schwarz & $S_H$, & $1\le i\le n-2$ & \checkmark & $V_n$ & \checkmark & \small{\ding{53}} & \checkmark & \checkmark & \checkmark & \checkmark & \checkmark& \small{\ding{53}}\\ \hline
Minkowski & $M_H$ &  & s\checkmark & $V_1$ & \checkmark & \checkmark & \checkmark & \checkmark & \checkmark & \small{\ding{53}} &\checkmark& \checkmark\\ \hline
Minkowski-Blaschke & $\overline{M}_H$, & $1\leq i\leq n-2$ & s\checkmark & $V_1$ & \checkmark & \small{\ding{53}} & \checkmark & \checkmark & \checkmark & \checkmark & \checkmark& \checkmark\\ \hline
Fiber & $F_{H}$, & $1\leq i\leq n-2$ & s\checkmark & \small{\ding{53}} & \checkmark & \checkmark & \checkmark & \checkmark & \checkmark & \small{\ding{53}} & \checkmark& \small{\ding{53}}\\ \hline
Inner & $I_H$ & $1\leq i\leq n-2$ & \checkmark & \small{\ding{53}} & \checkmark & \small{\ding{53}} & \checkmark & \checkmark & \checkmark & \checkmark & \checkmark& \small{\ding{53}}\\ \hline
Outer &$O_H$ & $1\leq i\leq n-2$  & s\checkmark & \small{\ding{53}} & \checkmark & \small{\ding{53}} & \checkmark & \checkmark & \checkmark & \checkmark & \checkmark& \checkmark\\ \hline
\end{tabular}
\end{center}
\vspace{.2in}
\caption{Properties, numbered as in Section~\ref{symm}, of the main symmetrizations,  where s\checkmark indicates strictly monotonic.}
\label{newtable}
\end{table}

Let ${\mathcal{B}}$ be a class of compact sets.  We can fix $i\in \{1,\dots,n-1\}$ but regard $\di$ as an entire collection of $i$-symmetrizations ${\di_H}:{\mathcal{B}}\rightarrow{\mathcal{B}}_H$, $H\in {\mathcal{G}}(n,i)$, and refer to $\di$ as an {\em $i$-symmetrization process}.  The definition of $\di_H$ may depend on $H$, so that, for example, $\di_H$ may be Steiner symmetrization for some $H$ and Minkowski symmetrization for others.  However, when we speak of familiar symmetrization processes, such as Steiner or Minkowski symmetrization, we assume that the type of symmetrization is the same for all $H$. We focus on the convergence of successive applications of $\di$ through a sequence of $i$-dimensional subspaces.  We shall use and extend ideas of Coupier and Davydov \cite{CouD14}, who consider only the case $i=n-1$, and adopt some of their notation in modified form.

Let $i\in \{1,\dots,n-1\}$ and suppose that $\di$ is an $i$-symmetrization process. Let $(H_m)$ be a sequence in ${\mathcal{G}}(n,i)$ and for convenience write $\di_m=\di_{H_m}$ for $m\in\N$. If $1\le l\le m$, let
\begin{equation}\label{dec121}
\di_{l,m}K=\di_{m}(\di_{{m-1}}(\cdots(\di_lK)\cdots))
\end{equation}
for each $K\in {\mathcal{B}}$, so that $\di_{l,m}K$ results from $m-l+1$ successive $\di$-symmetrizations applied to $K$ with respect to $H_l,H_{l+1},\dots,H_m$.

A sequence $(H_m)$ in ${\mathcal{G}}(n,i)$ is called {\em weakly $\di$-universal for ${\mathcal{B}}$} if for all $K\in {\mathcal{B}}$ and $l\in \N$, there exists $r(l,K)>0$ such that $\di_{l,m}K\to r(l,K)B^n$ as $m\to \infty$.  Note that this implies in particular that the successive symmetrals $\di_{1,m}K$ converge to a ball as $m\to \infty$.  If the constant $r(l,K)$ is independent of $l$, we say that $(H_m)$ is {\em $\di$-universal for ${\mathcal{B}}$}.  When ${\mathcal{B}}={\mathcal{K}}_n^n$, we shall use the shorter terms {\em weakly $\di$-universal} and {\em $\di$-universal} instead.  Example~\ref{weakdi} below exhibits a symmetrization process $\di$ and a sequence $(H_m)$ that is weakly $\di$-universal but not $\di$-universal.

When $i=n-1$, we shall use the terms in the previous paragraph also for the sequences $(u_m)$ of directions in $S^{n-1}$ such that $H_m=u_m^{\perp}$ for each $m$.  In fact, by \cite[Theorem~3.1]{CouD14}, a sequence $(u_m)$ in $S^{n-1}$ is Steiner-universal if and only if it is Minkowski-universal.   Since Steiner and Minkowski symmetrization preserve volume and mean width, respectively, it is easy to see that $(u_m)$ is weakly Steiner-universal (or weakly Minkowski-universal) if and only if it is Steiner-universal (or Minkowski-universal, respectively).  Much is known about such sequences; in particular, \cite[Proposition~3.3]{CouD14} implies that a sequence with each of these four equivalent properties can be drawn from any dense set in $S^{n-1}$.  See Section~\ref{compact} for further references.

\section{Extensions and variants of Klain's theorem}\label{Klainextension}

The goal of this section is to establish some extensions and variants of Klain's theorem \cite[Theorem~5.1]{Kla12} for Steiner symmetrization, the special case of Theorem~\ref{4} below corresponding to $i=n-1$.

Let ${\mathcal{B}}$ be a class of compact sets closed under intersections with $o$-symmetric balls and let $f:{\mathcal{B}}\to \R$ be a set function for which
\begin{equation}\label{layer}
\Omega_f(K)=\int_0^{\infty}f(K\cap rB^n)e^{-r^2}\,dr
\end{equation}
exists for each $K\in {\mathcal{B}}$.  We call $\Omega_f$ the {\em $f$-layering function} on ${\mathcal{B}}$. When ${\mathcal{B}}={\mathcal{K}}^n_n$ and $f=V_n$, we write $\Omega_{V_n}=\Omega$ for the usual layering function (see, for example, \cite[p.~344]{Kla12}).  Klain \cite[Theorem~3.1]{Kla12} proved the following result for Steiner symmetrization.

\begin{lem}\label{44}
Let ${\mathcal{B}}$ be a class of compact sets closed under intersections with $o$-symmetric balls and let $f:{\mathcal{B}}\to \R$ be an increasing set function.  Let $H\in {\mathcal{G}}(n,i)$, $i\in\{0,\dots,n-1\}$, and suppose that ${\di_H}:{\mathcal{B}}\rightarrow{\mathcal{B}}_H$ is monotonic, invariant on $o$-symmetric balls, and does not decrease $f$.  If $K\in {\mathcal{B}}$, then
\begin{equation}\label{omega}
\Omega_f(K)\le \Omega_f(\di_HK).
\end{equation}
If $f$ is strictly increasing on ${\mathcal{B}}$, then equality in \eqref{omega} implies that $\di_H(K\cap rB^n)=(\di_HK)\cap rB^n$ for each $r>0$.
\end{lem}

\begin{proof}
Since $\di_H$ is monotonic and invariant on $o$-symmetric balls, we have
\begin{equation}\label{om1}
\di_H(K\cap rB^n)\subset (\di_HK)\cap rB^n
\end{equation}
for all $r>0$.   Since $\di_H$ does not decrease $f$ and $f$ is increasing, it follows that
\begin{equation}\label{om2}
f(K\cap rB^n)\le f(\di_H(K\cap rB^n))\le f((\di_HK)\cap rB^n)
\end{equation}
for $r>0$ and (\ref{omega}) is then a consequence of the definition (\ref{layer}) of $\Omega_f$.

Suppose that $f$ is strictly increasing on ${\mathcal{B}}$, $K\in {\mathcal{B}}$, and equality holds in (\ref{omega}).  Then, in view of (\ref{om2}), we clearly have $f(K\cap rB^n)=f((\di_HK)\cap rB^n)$ for almost all $r>0$ and hence, since $f$ is increasing, for all $r>0$.  This implies that (\ref{om2}) holds with equality. Consequently, since $f$ is strictly increasing on ${\mathcal{B}}$, we deduce from (\ref{om1}) that $\di_H(K\cap rB^n)=(\di_HK)\cap rB^n$ for each $r>0$.
\end{proof}

\begin{ex}\label{44exa}
{\rm  The equality condition in Lemma~\ref{44} is not enough for our purposes:  We would like to have equality in (\ref{omega}) if and only if $\di_HK=K$.  However, this is not true in general, even when ${\mathcal{B}}={\mathcal{K}}^n_n$.  To see this, let $H\in {\mathcal{G}}(n,i)$, $i\in\{1,\dots,n-1\}$, and define $\di_H$ as follows. Choose $u\in S^{n-1}\cap H$ and let $R=\{tu: t\ge 0\}$. If $K\in {\mathcal{K}}^n$, let $C_r(K)$ be the spherical cap of $rS^{n-1}$ with center $ru$ and the same $(n-1)$-dimensional Hausdorff measure as $K\cap rS^{n-1}$, and let $A_HK=\cup\{C_r(K): r\ge 0\}$.  This process is sometimes called spherical cap symmetrization.  Then define $\di_HK=\conv A_HK$.

If $K\subset L$, we have $A_HK\subset A_HL$ and hence $\di_HK\subset \di_HL$, so $\di_H$ is monotonic.  It is clearly invariant on $o$-symmetric balls.  The construction of $A_HK$ ensures that $V_n(A_HK)=V_n(K)$, so $V_n(K)\le V_n(\di_HK)$ and $\di_H$ does not decrease $V_n$.  Thus $\di_H$ satisfies the hypotheses of Lemma~\ref{44} with ${\mathcal{B}}={\mathcal{K}}^n_n$ and $f=V_n$.

Now let $x\not\in H$ and let $K=B^n+x$. Then $\di_HK$ is an $H$-symmetric translate of $B^n$, so $\di_HK\neq K$.  On the other hand, $V_n((\di_HK)\cap rB^n)=V_n(K\cap rB^n)$ for each $r\ge 0$, and hence $\Omega(\di_HK)=\Omega(K)$.\qed
}
\end{ex}

Note that the $i$-symmetrization $\di_H$ from Example~\ref{44exa} is also idempotent and continuous, but it is neither invariant on $H$-symmetric cylinders nor invariant under translations orthogonal to $H$ of $H$-symmetric sets.  Since the symmetral of a ball is an $H$-symmetric ball, it is easy to see that Klain's theorem \cite[Theorem~5.1]{Kla12} is not true for the corresponding $i$-symmetrization process $\di$.

Klain \cite[Theorem~3.1]{Kla12} proved the following result for Steiner symmetrization, which corresponds to the special case when $i=n-1$.

\begin{lem}\label{44a}
Let $i\in\{1,\dots,n-1\}$.  When $\di_H=F_H$, fiber symmetrization with respect to $H$, equality holds in \eqref{omega} with ${\mathcal{B}}={\mathcal{K}}^n_n$ and $f=V_n$ if and only if $F_HK=K$.  Hence the corresponding conclusion also holds when $\di_H$ is monotonic, invariant on $H$-symmetric sets, and invariant under translations orthogonal to $H$ of $H$-symmetric sets.  In particular, it holds when $\di_H=M_H$, Minkowski symmetrization with respect to $H$.
\end{lem}

\begin{proof}
Let $K\in {\mathcal{K}}^n$, let $r>0$, and let $x\in H$.  If equality holds in (\ref{omega}) with ${\mathcal{B}}={\mathcal{K}}^n_n$ and $f=V_n$, then by Lemma~\ref{44}, we have $F_H(K\cap rB^n)=(F_HK)\cap rB^n$.  Using the definition (\ref{fhjk}) of $F_H$ and the Brunn-Minkowski inequality in $H^{\perp}+x$, we obtain
\begin{eqnarray*}
\lefteqn{V_{n-i}((F_HK)\cap rB^n\cap(H^{\perp}+x))^{1/(n-i)}=
 V_{n-i}(F_H(K\cap rB^n)\cap(H^{\perp}+x))^{1/(n-i)}}\\
&=& V_{n-i}\left(\frac12(K\cap rB^n\cap (H^{\perp}+x))+\frac12R_H(K\cap rB^n\cap (H^{\perp}+x))\right)^{1/(n-i)}\\
&\ge &
\frac12 V_{n-i}(K\cap rB^n\cap (H^{\perp}+x))^{1/(n-i)}+
\frac12 V_{n-i}(R_H(K\cap rB^n\cap (H^{\perp}+x)))^{1/(n-i)}\\
&=& V_{n-i}(K\cap rB^n\cap (H^{\perp}+x))^{1/(n-i)}
\end{eqnarray*}
and hence
\begin{equation}\label{fibeq}
V_{n-i}((F_HK)\cap rB^n\cap(H^{\perp}+x))\ge V_{n-i}(K\cap rB^n\cap (H^{\perp}+x)).
\end{equation}
Equality holds if and only if $K\cap rB^n\cap (H^{\perp}+x)$ is homothetic to $R_H(K\cap rB^n\cap (H^{\perp}+x))$ and hence if and only if $K\cap rB^n\cap (H^{\perp}+x)$ is symmetric in $H^{\perp}+x$ with respect to $x$.

Suppose that $F_HK\neq K$.  Then, since $F_H$ is invariant on $H$-symmetric sets, $K$ is not $H$-symmetric.  It follows that there is an $r>0$ such that $K\cap rB^n$ is not $H$-symmetric and hence an $x\in H$ such that $K\cap rB^n\cap (H^{\perp}+x)$ is not symmetric in $H^{\perp}+x$ with respect to $x$.   From the previous paragraph, we conclude that strict inequality holds in (\ref{fibeq}).  By Fubini's theorem, we obtain $V_n((F_HK)\cap rB^n)>V_n(K\cap rB^n)$ and so, by continuity and (\ref{layer}), $\Omega(F_HK)>\Omega(K)$.  This proves the first claim in the lemma.

Suppose that $\di_H$ is monotonic, invariant on $H$-symmetric sets, and invariant under translations orthogonal to $H$ of $H$-symmetric sets. By \cite[Corollary~7.3]{BGG}, $F_HK\subset \di_HK$.  Therefore $\Omega(K)\le \Omega(F_HK)\le \Omega(\di_HK)$, by Lemma~\ref{44}.  If $\Omega(\di_HK)=\Omega(K)$, then $\Omega(F_HK)=\Omega(K)$.  Consequently, $F_HK=K$, implying that $K$ is $H$-symmetric and hence that $\di_HK=K$.
\end{proof}

\begin{lem}\label{44b}
Let $i\in\{1,\dots,n-2\}$.  When $\di_H=S_H$, Schwarz symmetrization with respect to $H$, equality holds in \eqref{omega} with ${\mathcal{B}}={\mathcal{K}}^n_n$ and $f=V_n$ if and only if $S_HK=K$.
\end{lem}

\begin{proof}
Let $K\in {\mathcal{K}}^n$ and let $H\in {\mathcal{G}}(n,i)$.  By identifying $H^{\perp}$ with $\R^{n-i}$, choose a Steiner-universal sequence $(H_m)$ such that $H_m\in {\mathcal{G}}(n,n-i-1)$ with $H_m\subset H^{\perp}$, $m\in \N$.  Let $J_m=H+H_m\in {\mathcal{G}}(n,n-1)$, $m\in \N$.  Then the Steiner symmetrals $S_{J_m}K$ converge to $S_HK$ as $m\to\infty$.

Suppose that equality holds in (\ref{omega}) with $f=V_n$, i.e., $\Omega(S_HK)=\Omega(K)$. Let $K_m= S_{J_m}\cdots S_{J_1}K$, $m\in \N$.  By Lemma~\ref{44}, applied successively with $K$ and $\di_H$ replaced by $K_k$ and $S_{J_{k+1}}$, and the continuity of $\Omega$, we obtain
$$\Omega(K)\le \Omega(K_1)\le \Omega(K_2)\le \cdots\le \Omega(K_m)\to \Omega(S_HK)=\Omega(K).$$
Since equality must hold throughout, successive use of Lemma~\ref{44a} with $i=n-1$ (in which case $F_H=S_H$) now yields
$$K=K_1=K_2=\cdots=K_m=\cdots=S_HK,$$
as required.
\end{proof}

\begin{lem}\label{25June}
Let $i\in\{1,\dots,n-2\}$.  When $\di_H=\overline{M}_HK$, equality holds in \eqref{omega} with ${\mathcal{B}}={\mathcal{K}}^n_n$ and $f=V_n$ if and only if $\overline{M}_HK=K$.
\end{lem}

\begin{proof}
We have $\Omega(K)\le \Omega(S_HK)\le \Omega(\overline{M}_HK)$, by Lemma~\ref{44} (with $\di_H=S_H$ and $f=V_n$) and Lemma~\ref{1}.  If $\Omega(\overline{M}_HK)=\Omega(K)$, then $\Omega(S_HK)=\Omega(K)$.  Consequently, $S_HK=K$, by Lemma~\ref{44b}, implying that $K$ is rotationally symmetric with respect to $H$ and hence that $\overline{M}_HK=K$.
\end{proof}

The following result generalizes Klain's theorem \cite[Theorem~5.1]{Kla12} for Steiner symmetrization, which corresponds to the special case when $i=n-1$.  The proof uses ideas from \cite{BBGV}.

\begin{thm}\label{4}
Let $i\in\{1,\dots,n-1\}$ and let  $(H_m)$ be a sequence chosen from a finite set ${\mathcal{F}}=\{U_1,\dots,U_k\}\subset {\mathcal{G}}(n,i)$. Then, for every $K\in {\mathcal{K}}^n$, the successive fiber symmetrals
\begin{equation}\label{mseq}
K_m= F_{H_m}\cdots F_{H_1}K
\end{equation}
converge to a compact convex set $L$ as $m\to\infty$.  Furthermore, $L$ is $U_j$-symmetric for each $U_j\in {\mathcal{F}}$ that appears infinitely often in $(H_m)$.
\end{thm}

\begin{proof}
A few preliminary observations will be useful.
Let $H\in {\mathcal{G}}(n,i)$ and let $K,L\in {\mathcal{K}}^n$, where $L$ is $H$-symmetric. Note that $F_H$ does not decrease volume.  Indeed, the definition (\ref{fhjk}) of $F_H$ shows that it is equivalent to Minkowski symmetrization in each $(n-i)$-dimensional plane orthogonal to $H$.  It follows that $$V_{n-i}(K\cap (H^{\perp}+x))\le V_{n-i}((F_HK)\cap (H^{\perp}+x))$$
for each $x\in H$, since Minkowski symmetrization preserves $V_1$ and does not decrease $V_j$ for $j>1$ (see \cite[p.~58]{BGG}).  Then Fubini's theorem yields
\begin{equation}\label{thmom1}
V_n(K)\leq V_n(F_HK).
\end{equation}
Furthermore
\begin{equation}\label{thmom2}
V_n(K\cap L)\leq V_n((F_HK)\cap L)
\end{equation}
holds because
$$F_H(K\cap L)\subset (F_HK)\cap F_HL=(F_HK)\cap L,$$
since $F_H$ is monotonic and invariant on $H$-symmetric sets. Hence, by (\ref{thmom1}) with $K$ replaced by $K\cap L$,
$$V_n(K\cap L)\le V_n(F_H(K\cap L))\le V_n((F_HK)\cap L).$$

We follow Klain's argument. Dropping an initial segment $(K_m)_{m\le N}$ of the sequence $(K_m)$ defined by (\ref{mseq}) and possibly replacing ${\mathcal{F}}$ by one of its subsets, we may  assume, without loss of generality,  that each subspace in ${\mathcal{F}}$ appears infinitely often in the sequence $(H_m)$. The main idea is to construct a subsequence along which the subspaces $U_j\in {\mathcal{F}}$ appear in a particular order. With each index $m$, we associate a permutation $\pi_m$ of $\{1,\dots, k\}$ that indicates the order in which the subspaces $U_1,\dots,U_k$ appear for the first time among those $U_j$ with $j\geq m$.  Since there are only finitely many permutations, we can pick a subsequence $(H_{m_p})$ such that the permutation $\pi_{m_p}$ is the same for each $p$. By relabeling the subspaces, we may assume that this permutation is the identity.  Passing to a further subsequence, we may assume that every subspace in ${\mathcal{F}}$ appears in each segment $H_{m_p}, H_{m_p+1},\dots, H_{m_{p+1}-1}$.

By Blaschke's selection theorem, there is a subsequence (again denoted by $(K_{m_p})$) that converges in the Hausdorff metric to some
$L\in {\mathcal{K}}^n$.  We note for later use that the entire sequence $\bigl(V_n(K_m)\bigr)$ is increasing, by (\ref{thmom1}), and hence convergent.
Moreover,
\begin{equation}\label{mono-0}
\lim_{m\to\infty} V_n(K_m)=\lim_{p\to\infty}
V_n(K_{m_p}) =V_n(L).
\end{equation}

Now assume that $K\in {\mathcal{K}}^n_n$.

We show by induction that $L$ is $U_j$-symmetric for $j=1,\dots,k$. For $j=1$, observe that $H_{m_p}=U_1$ for each $p$. Therefore $K_{m_p}$ is $U_1$-symmetric for each $p$ and the same is true for $L$.  Suppose that $L$ is $U_r$-symmetric for $r=1,\dots,j-1$. Let $m'_p$ be the index where $U_j$ appears for the first time after $H_{m_p}$.  Then for $m_p+1\le m\le m'_p-1$, $H_m=U_r$ for some $r=1,\dots,j-1$, so we can apply the inductive hypothesis and (\ref{thmom2}), successively with $H$ equal to one of the latter subspaces, to obtain
\begin{equation}\label{eq:Klain-1}
V_n(K_{m_p}\cap L)\le V_n((F_{H_{m'_p-1}}\cdots F_{H_{m_p+1}}K_{m_p})\cap L)=V_n(K_{m'_p-1}\cap L).
\end{equation}
Since $(K_{m_p})$ converges to $L$, (\ref{mono-0}) implies that the left-hand side of (\ref{eq:Klain-1}) converges to $V_n(L)$. Therefore the right-hand side converges to $V_n(L)$, which implies that
$$V_n(K_{m'_p-1}\Delta L)=V_n(K_{m'_p-1})+V_n(L)-2V_n(K_{m'_p-1}\cap L)$$
converges to zero.  Then, because $K_{m'_p-1}\in {\mathcal{K}}^n_n$, it follows that $K_{m'_p-1}$ converges to $L$ as $p\to\infty$.  Now we use this, Lemma~\ref{44} (with ${\mathcal{B}}={\mathcal{K}}^n$, $f=V_n$, and $\di_H=F_H$), the continuity of fiber symmetrization, and the continuity of the functional $\Omega$, to obtain
\begin{equation}\label{eq:mono2}
\Omega(L)\le\Omega(F_{U_j}L)=\lim_{p\to\infty} \Omega(F_{U_j}K_{m'_p-1})
=\lim_{p\to\infty} \Omega(K_{m'_p}).
\end{equation}
Since $\bigl(\Omega (K_m)\bigr)$ is an increasing sequence, by Lemma~\ref{44}, and since it contains the subsequence $\bigl(\Omega(K_{m_p})\bigr)$ which converges to $\Omega(L)$ because $\Omega$ is continuous, the first and last term in \eqref{eq:mono2} are equal, so equality holds throughout.  By Lemma~\ref{44a}, $F_{U_j}L=L$.  Therefore $L$ is $U_j$-symmetric and this concludes the inductive step.

It remains to prove that the entire sequence converges. Since $L$ is $U_j$-symmetric for $j=1,\dots,k$, we have, by the same reasoning as in \eqref{eq:Klain-1} and the lines following it, that
$$V_n(K_{m_p}\cap L)\le V_n(K_m\cap L),$$
for each $m\geq m_p$; moreover, since $V_n(K_{m_p}\cap L)\to V_n(L)$ as $p\to\infty$, we also have $V_n(K_m\cap L)\to V_n(L)$ as $m\to\infty$.  Since $K_m\in {\mathcal{K}}^n_n$, this implies that $K_m$ converges to $L$ as $m\to\infty$.  This completes the proof when $K\in {\mathcal{K}}^n_n$.

Suppose that $\dim K<n$.  Fiber symmetrization is invariant under translations, so we may assume that $o\in K$.  Then $o\in K_m$ for all the successive symmetrals $K_m$ of $K$.  Since fiber symmetrization with respect to a subspace $H$ corresponds to Minkowski symmetrization in each plane $H^{\perp}+x$, $x\in H$, it follows that $\aff K_m\subset \aff K_{m+1}$ for all $m$.  Then there is an $M\in \N$ such that $\aff K_m=\aff K_M$ for all $m\ge M$.  Consequently, we may as well assume, by replacing $K$ by $K_M$ if necessary, that each $K_m$ is contained in a subspace $S$ and $\dim K_m=\dim S=k$, say.  The cases $k=0$ and $k=1$ are trivial, and if $k\ge 2$, the previous argument can be repeated, with $n$ replaced by $k$, by identifying $S$ and $\R^k$.
\end{proof}

\begin{thm}\label{corthm4}
Let $i\in \{1,\dots,n-1\}$ and let $\di$ be an $i$-symmetrization process on ${\mathcal{K}}^n_n$.  Suppose that for each $H\in {\mathcal{G}}(n,i)$,
$\di_H$ is monotonic, invariant on $H$-symmetric sets, invariant under translations orthogonal to $H$ of $H$-symmetric sets, and continuous.  Then Klain's theorem \cite[Theorem~5.1]{Kla12} holds for $\di$.  In particular, it holds for Minkowski symmetrization, and in this case, the result also applies to $K\in {\mathcal{K}}^n$.
\end{thm}

\begin{proof}
We first check that the proof of Theorem~\ref{4} for the case when $K\in {\mathcal{K}}^n_n$ works when fiber symmetrization is replaced by an $i$-symmetrization process $\di$ with the four stated properties.  Indeed, the first three properties imply that for each $H\in {\mathcal{G}}(n,i)$, $F_HK\subset \di_HK$, by \cite[Corollary~7.3]{BGG}.  Then (\ref{thmom1}) and (\ref{thmom2}) clearly hold.  The use of Blaschke's selection theorem requires only that the successive symmetrals are uniformly bounded, and this is ensured by the invariance on $H$-symmetric sets.  No further assumptions are needed except for (\ref{eq:mono2}), which holds for $\di$ by the continuity hypothesis.  Lemma~\ref{44a} allows the conclusion that $\di_{U_j}L=L$ at the end of the inductive step and the rest of the proof is straightforward.

Minkowski symmetrization has all the properties stated in the theorem.  In this case the last paragraph of the proof of Theorem~\ref{4} applies without change if fiber symmetrization is replaced by Minkowski symmetrization.
\end{proof}

\begin{ex}\label{44exc}
{\rm (cf.~\cite[Example~5.2]{BGG}.) For all $K\in {\mathcal{K}}^n_n$ and $H\in{\mathcal{G}}(n,i)$, let $\di K$ be the smallest $H$-symmetric spherical cylinder such that some translate orthogonal to $H$ contains $K$.  Then ${\di}_H$ is monotonic, invariant under translations orthogonal to $H$ of $H$-symmetric sets, and continuous, but not invariant on $H$-symmetric sets. Klain's theorem \cite[Theorem~5.1]{Kla12} is not true for the corresponding $i$-symmetrization process $\di$.  To be specific, let $n=2$, $i=1$, and for $m\in \N$, let $H_{2m+1}=(1,1)^{\perp}$ and $H_{2m}=(0,1)^{\perp}$.  If $K=[-1,1]^2$, the successive symmetrals $K_m=\di_{H_m}\cdots\di_{H_1}K$ do not converge; indeed, they are not even uniformly bounded.
\qed}
\end{ex}

\begin{ex}\label{28June}
{\rm (cf.~\cite[Example~5.14]{BGG}.) Let $K\in {\mathcal{K}}^n_n$ and let $H\in{\mathcal{G}}(n,i)$.  If $K=L+y$, where $L$ is $H$-symmetric and $y\in H^{\perp}$, then define $\di K=L$.  Otherwise, define $\di K=t_KC^n$, where $C^n$ is an $o$-symmetric cube with $V_n(C^n)=1$ and a facet parallel to $H$ (and hence $H$-symmetric), and where $t_K\ge 0$ is chosen so that $V_n(\di K)=V_n(K)$. Then ${\di}_H$ is invariant under translations orthogonal to $H$ of $H$-symmetric sets and invariant on $H$-symmetric sets, but neither monotonic nor continuous.  Klain's theorem \cite[Theorem~5.1]{Kla12} is not true for the corresponding $i$-symmetrization process $\di$.  In fact, let $n=2$, $i=1$, $0<\theta<\pi/4$, and for $m\in \N$, let $H_{2m+1}=(\cos\theta,\sin\theta)^{\perp}$ and $H_{2m}=(0,1)^{\perp}$.  If $K=[-1,1]^2$, the successive symmetrals $K_m=\di_{H_m}\cdots\di_{H_1}K$ do not converge.  To see this, note that $K$ is not a translate orthogonal to $H_1$ of an $H_1$-symmetric set, so $K_1=\di_{H_1}K=\phi K$, where $\phi$ denotes a rotation by $\theta$ around the origin.  Then $K_1$ is not a translate orthogonal to $H_2$ of an $H_2$-symmetric set, so $K_2=\di_{H_2}K_1=\phi^{-1}K_1=K$.  It follows that $K_m=K$ for even $m$ and $K_m=\phi K$ for odd $m$.\qed}
\end{ex}

Despite the previous two examples, it is possible that the assumptions in the previous theorem can be weakened; see Problem~\ref{prob0}.  In particular, we do not know if the continuity of $\di$ is necessary, although the following example shows that it is not a consequence of the other assumptions in Theorem~\ref{corthm4}.

\begin{ex}\label{44exd}
{\rm Let $H\in{\mathcal{G}}(n,n-1)$ and for $K\in {\mathcal{K}}^n_n$, let $\di_HK=S_HK$ if $V_n(K)<1$, and let $\di_HK=M_HK$ if $V_n(K)\ge 1$.  Then $\di_H$ is monotonic, invariant on $H$-symmetric sets, and invariant under translations orthogonal to $H$ of $H$-symmetric sets, but not continuous. \qed}
\end{ex}

\begin{thm}\label{thm5cor}
Klain's theorem \cite[Theorem~5.1]{Kla12} (cf.~Theorem~\ref{4}) holds for Schwarz symmetrization, where the limit set $L$ is rotationally symmetric with respect to each $U_j\in {\mathcal{F}}$ that appears infinitely often in $(H_m)$.  The same is true for Minkowski-Blaschke symmetrization.
\end{thm}

\begin{proof}
We follow the proof of Theorem~\ref{4} for the case when $K\in {\mathcal{K}}^n_n$, replacing fiber symmetrization by either Schwarz symmetrization or Minkowski-Blaschke symmetrization.  Throughout, we replace symmetry with respect to a subspace by rotational symmetry with respect to the subspace.

For the first statement in the theorem, note that Schwarz symmetrization is monotonic, continuous, and does not increase (in fact preserves) $V_n$.  Then the proof goes through without difficulty if Lemma~\ref{44b} is used instead of Lemma~\ref{44a}.

For the second statement, we use Lemma~\ref{1} to obtain $S_HK\subset \overline{M}_HK$ for each $K\in {\mathcal{K}}^n_n$ and $H\in {\mathcal{G}}(n,i)$.  This and the fact that
(\ref{thmom1}) and (\ref{thmom2}) hold when $F_H$ is replaced by $S_H$ allow us to conclude that (\ref{thmom1}) and (\ref{thmom2}) also hold when $F_H$ is replaced by $\overline{M}_H$.  The use of Blaschke's selection theorem requires only that the successive symmetrals are uniformly bounded, and this is ensured by the containment $\overline{M}_HK\subset O_HK$ from Lemma~\ref{1}.  No further assumptions are needed except for (\ref{eq:mono2}), which holds because $\overline{M}_H$ is continuous.  Lemma~\ref{25June} allows the conclusion that $\overline{M}_{U_j}L=L$ at the end of the inductive step and the rest of the proof is straightforward.
\end{proof}

The first statement in the previous theorem extends to the Schwarz symmetrization of compact sets; see Theorem~\ref{july24thm}.

\section{Universal and weakly-universal sequences}\label{universal}

In \cite[Corollary~5.4]{Kla12}, Klain proves the following result:  {\it Let $v_1,\dots,v_k$ be a set of directions in $\R^n$ that contains an irrational basis for $\R^n$.  Suppose that $(u_m)$ is a sequence of directions such that each $u_m$ belongs to $\{v_1,\dots,v_k\}$ and each element of the irrational basis appears infinitely often in $(u_m)$. Then $(u_m)$ is weakly Steiner-universal.}

Weakly Steiner-universal is equivalent to Steiner-universal, as we know.  Thus Klain's result provides specific Steiner-universal sequences, the novel feature being that only a finite set of directions are used.

\begin{thm}\label{21June}
Let $i\in\{1,\dots,n-1\}$ and let $U_j\in {\mathcal{G}}(n,i)$, $j=1,\dots,k$, be as in Corollary~\ref{cor21June} (with $H_j$ replaced by $U_j$). Let $(H_m)$ be a sequence chosen from $\{U_1,\dots,U_k\}$ in which each $U_j$ appears infinitely often.  Then $(H_m)$ is Minkowski-universal.
\end{thm}

\begin{proof}
Let $K\in{\mathcal{K}}^n_n$ and let $l\in \N$.  By Theorem~\ref{corthm4}, the result $K_m$ of successive Minkowski symmetrizations of $K$ with respect to $H_l, H_{l+1},\dots,H_m$, $m\ge l$, converges as $m\to\infty$ to an $L\in {\mathcal{K}}^n_n$ such that $R_{U_j}L=L$ for $j=1,\dots,k$.  By Corollary~\ref{cor21June} with $K$ and $H_j$ replaced by $L$ and $U_j$, respectively, $L=rB^n$ for some $r=r(l,K)>0$.  Since Minkowski symmetrization preserves mean width, $V_1(rB^n)=V_1(K)$, proving that $r$ is independent of $l$.
\end{proof}

Note that the subspaces $U_j$ in the previous theorem are those specified in Theorem~\ref{teo:sphericalsym_lines} (for $i=1$), Theorem~\ref{finthm} (for $1<i<n-1$), and \cite[Proposition~4.2]{BCD} (for $i=n-1$).  Since Minkowski-universal sequences are Steiner-universal when $i=n-1$, Theorem~\ref{21June} is an extension of Klain's result \cite[Corollary~5.4]{Kla12} stated above.

The following result is \cite[Theorem~8.1]{BGG}.

\begin{prop}\label{Successive}
Let $i\in \{1,\dots,n-1\}$ and let $\di$ be an $i$-symmetrization process on ${\mathcal{K}}^n_n$.  Suppose that
$$
I_HK\subset\di_HK\subset M_HK
$$
for all $H\in {\mathcal{G}}(n,i)$ and all $K\in {\mathcal{K}}^n_n$.  If $(H_m)$ is a Minkowski-universal sequence in ${\mathcal{G}}(n,i)$, then $(H_m)$ is weakly $\di$-universal.
\end{prop}

Note that if $i=n-1$, then $I_H=S_H$, the Steiner symmetral.  As was observed above, Minkowski-universal and Steiner-universal sequences coincide.  By \cite[Theorem~6.3]{BGG}, when $i=n-1$ the hypotheses of the following corollary hold if the assumption that $\di_H$ is monotonic and invariant under $H$-symmetric sets is replaced by the assumption that $\di_H$ is strictly monotonic, idempotent, and either invariant on $H$-symmetric spherical cylinders or projection invariant.  The next result is \cite[Corollary~8.2]{BGG}.

\begin{prop}\label{Successivecor}
Let $i\in \{1,\dots,n-1\}$ and let $\di$ be an $i$-symmetrization process on ${\mathcal{K}}^n_n$.  Suppose that for each $H\in {\mathcal{G}}(n,i)$,
${\di_H}$ is monotonic, invariant on $H$-symmetric sets, and invariant under translations orthogonal to $H$ of $H$-symmetric sets.  If $(H_m)$ is a Minkowski-universal sequence in ${\mathcal{G}}(n,i)$, then $(H_m)$ is weakly $\di$-universal.
\end{prop}

In particular, the sequence $(H_m)$ from Theorem~\ref{21June} is weakly $\di$-universal whenever $\di$ satisfies the hypotheses of Proposition~\ref{Successivecor}, providing a further extension of Klain's result \cite[Corollary~5.4]{Kla12} stated above.

\cite[Examples~5.10 and~5.14]{BGG}, both with $j=1$ (say) and $B^n$ replaced by an $H$-symmetric $n$-dimensional cube, show that the assumptions of invariance on $H$-symmetric sets and monotonicity, respectively, cannot be dropped in Proposition~\ref{Successivecor}.  We do not have an example showing that the assumption of invariance under translations orthogonal to $H$ of $H$-symmetric sets is necessary (see Problem~\ref{prob0}).  However, the following example (see \cite[Example~8.3]{BGG}) shows that if this assumption is omitted, the hypotheses of Proposition~\ref{Successivecor} do not allow the stronger conclusion that $(H_m)$ is $\di$-universal.

\begin{ex}\label{weakdi}
{\em  Let $\di$ be the symmetrization process corresponding to the symmetrization $\di_H$ in \cite[Example~5.8]{BGG}, with $n=2$ and $i=1$.  Let $0<\theta<\pi/2$ be an irrational multiple of $\pi$ and let $H_1$ be the line through the origin in the direction $(\cos\theta,\sin\theta)$.
For $m\in \N$, let $H_{2m+1}=H_1$ and $H_{2m}=(0,1)^{\perp}$.  Then the sequence $(H_m)$ is Steiner-universal; see, for example, \cite[Corollary~5.4]{Kla12}.  It is shown in \cite[Example~8.3]{BGG} that $(H_m)$ is weakly $\di$-universal but not $\di$-universal.
\qed}
\end{ex}

We now examine Schwarz and Minkowski-Blaschke symmetrization.

\begin{ex}\label{ex21}
{\em Define $\di_H:{\mathcal{K}}^n_n\rightarrow \left({\mathcal{K}}^n_n\right)_H$ by
$$\di_HK=(K|H)+t_K(B^n\cap H^{\perp}),$$
where $t_K\ge 0$ is chosen so that $V_n(\di_HK)=V_n(K)$.  Then $\di_H$ is volume-preserving, idempotent, invariant on $H$-symmetric spherical cylinders, projection invariant, invariant under translations orthogonal to $H$ of $H$-symmetric sets, and rotationally symmetric, but not monotonic or invariant on $H$-symmetric sets.\qed}
\end{ex}

\begin{ex}\label{ex22}
{\em (Generalized Schwarz symmetrization.) Let $i\in \{1,\dots,n-2\}$, let $k>0$, and let $f:{\mathcal{K}}^{n-i}_{n-i}\to [0,\infty)$ be strictly increasing, rigid-motion invariant, homogeneous of degree $k$, and such that
\begin{equation}\label{BrunnM}
f((1-t)K+tL)^{1/k}\ge (1-t)f(K)^{1/k}+tf(L)^{1/k}
\end{equation}
for $t\in [0,1]$ and $K,L\in {\mathcal{K}}^{n-i}_{n-i}$.  Let $H\in {\mathcal{G}}(n,i)$.  For $K\in {\mathcal{K}}^n_n$, define $\di_HK$ such that for each $(n-i)$-dimensional plane $G$ orthogonal to $H$ and meeting $K$, the set $G\cap \di_HK$ is a (possibly degenerate) $(n-i)$-dimensional closed ball with center in $H$ such that $f(G\cap \di_HK)=f(G\cap K)$.  A standard argument for the Brunn-Minkowski inequality (see \cite[p.~361]{Gar02}) shows that (\ref{BrunnM}) and the homogeneity of $f$ imply that $\di_HK\in {\mathcal{K}}^n_n$.  Then $\di_H$ is strictly monotonic, idempotent, invariant on $H$-symmetric spherical cylinders, projection invariant, invariant under translations orthogonal to $H$ of $H$-symmetric sets, and rotationally symmetric, but not invariant on $H$-symmetric sets. \qed}
\end{ex}

When $f=V_{n-i}$ in Example~\ref{ex22}, we retrieve the classical Schwarz symmetrization.  One can also take $f=V_j$, $j\in \{1,\dots,n-i-1\}$, since (\ref{BrunnM}) is then the Brunn-Minkowski inequality for quermassintegrals \cite[(74), p.~393]{Gar02} (see also \cite[Satz~XI, p.~260]{Had57}); in this case, $\di_H$ is not $V_k$-preserving for any $k\in \{1,\dots,n\}$.  In view of this, the following characterization of Schwarz symmetrization is worth stating.  Note that by \cite[Theorem~3.2]{Sar}, the assumption that $\di$ is invariant on $H$-symmetric cylinders can be replaced by projection invariance.

\begin{thm}\label{Schwarz}
Let $i\in \{1,\dots,n-2\}$, let $H\in {\mathcal{G}}(n,i)$, and let $\di_H$ be an $i$-symmetrization on ${\mathcal{K}}^n_n$.  Suppose that $\di_H$ is monotonic, volume preserving, rotationally symmetric, and invariant on $H$-symmetric cylinders.  Then $\di_H$ is Schwarz symmetrization with respect to $H$.
\end{thm}

\begin{proof}
If $\di_H$ is monotonic, volume preserving, and invariant on $H$-symmetric cylinders, then by \cite[Theorem~10.1(i)]{BGG}, we have
\begin{equation}\label{vni}
V_{n-i}\left((\di K)\cap (H^{\perp}+x)\right)=V_{n-i}\left(K\cap (H^{\perp}+x)\right)
\end{equation}
for all $K\in {{\mathcal{K}}^n_n}$ and $x\in H$.  Then we need only observe that if $\di_H$ is rotationally symmetric and (\ref{vni}) is satisfied, $\di_H$ must be Schwarz symmetrization with respect to $H$ by its very definition.
\end{proof}

Examples~\ref{ex21}, \ref{ex22}, \cite[Example~10.7]{BGG}, and \cite[Example~5.10]{BGG} with $j=n$, show that none of the assumptions in the previous theorem can be omitted.

The following corollary extends \cite[Theorem~3.1]{CouD14}, which corresponds to the case when $i=n-1$.

\begin{cor}\label{CD}
Let $i\in \{1,\dots,n-1\}$. A sequence $(H_m)$ in ${\mathcal{G}}(n,i)$ is Schwarz-universal if and only if it is Minkowski-Blaschke-universal.
\end{cor}

\begin{proof}
We know that $S_{H_m}$ preserves $V_n$ and does not increase $V_1$.  We also know that $\overline{M}_{H_m}$ preserves $V_1$ and conclude from Lemma~\ref{1} that $\overline{M}_{H_m}$ does not decrease $V_n$.  This allows the proof of \cite[Theorem~3.1]{CouD14} to be applied almost verbatim.
\end{proof}

The special case of the following theorem when $i=1$, $n=3$, $k=2$, $U_1$ and $U_2$ are two orthogonal lines through the origin in $\R^3$, $H_{2m+1}=H_1$, and $H_{2m}=H_2$ for $m\in \N$, was first proved by Tonelli \cite{Ton15}. His rather long and complicated argument applied not only to convex bodies but general compact sets.  Tsolomitis \cite[Theorem~1.7(ii)]{Tso96} proves that Schwarz-universal sequences exist for some other values of $i$ and $n$.  However, his result requires that $n^{2/(n-1)}<2$, which only holds when $n\ge 7$.

\begin{thm}\label{22}
Let $i\in\{1,\dots,n-2\}$ and let $U_j\in {\mathcal{G}}(n,i)$, $j=1,\dots,k$, be as in Theorem~\ref{new}(i) (with $H_j$ replaced by $U_j$).   Let $(H_m)$ be a sequence chosen from $\{U_1,\dots,U_k\}\subset {\mathcal{G}}(n,i)$ in which each $U_j$ appears infinitely often. Then $(H_m)$ is Schwarz-universal and hence, by Corollary~\ref{CD}, also Minkowski-Blaschke-universal.
\end{thm}

\begin{proof}
Let $K\in{\mathcal{K}}^n_n$ and let $l\in \N$.  By Theorem~\ref{thm5cor}, the result $K_m$ of successive Schwarz symmetrizations of $K$ with respect to $H_l, H_{l+1},\dots,H_m$, $m\ge l$, converges as $m\to\infty$ to an $L\in {\mathcal{K}}^n_n$ that is rotationally symmetric with respect to $U_j$ for $j=1,\dots,k$.  By Corollary~\ref{cor22June} with $K$ and $H_j$ replaced by $L$ and $U_j$, respectively, $L=rB^n$ for some $r=r(l,K)>0$.  Since Schwarz symmetrization preserves volume, $V_n(rB^n)=V_n(K)$, proving that $r$ is independent of $l$.
\end{proof}

Note that in view of Lemma~\ref{1}, we can take $\di_H=\overline{M}_H$ in the next theorem. The following proof is essentially the same as that of \cite[Theorem~8.1]{BGG} but is included for the reader's convenience.

\begin{thm}\label{2}
Let $i\in \{1,\dots,n-1\}$ and let $\di$ be an $i$-symmetrization process on ${\mathcal{K}}^n_n$.  Suppose that
\begin{equation}\label{SdiM2}
S_HK\subset\di_HK\subset O_HK
\end{equation}
for all $H\in {\mathcal{G}}(n,i)$ and all $K\in {\mathcal{K}}^n_n$.  If $(H_m)$ is a Schwarz-universal sequence in ${\mathcal{G}}(n,i)$, then $(H_m)$ is weakly $\di$-universal.
\end{thm}

\begin{proof}  Let $(H_m)$ be Schwarz-universal and let $K\in {\mathcal{K}}^n_n$.  Using the right-hand containment in (\ref{SdiM2}), is easy to see that any ball with center at the origin that contains $K$ will also contain all the successive $\di$-symmetrals $\di_{1,m}K$. If $m\in \N$ and $L\in {\mathcal{K}}^n_n$, then by (\ref{SdiM2}), we have $\di_mL\supset S_mL$ and hence $V_n(\di_mL)\ge V_n(S_mL)=V_n(L)$, where $S$ stands for Schwarz symmetrization.  Taking $L=\di_{1,m-1}K$, we obtain $V_n(\di_{1,m}K)\ge V_n(\di_{1,m-1}K)$ for all $m=2,3,\dots$.  Therefore $V_n(\di_{1,m}K)\to a>0$, say, as $m\to \infty$.

By Blaschke's selection theorem, there is a subsequence $(H_{m_p})$ of $(H_m)$ such that $\di_{1,m_p}K\to J\in{\mathcal{K}}^n_n$ as $p\to\infty$, where $V_n(J)=a$.

Now if $1\le p \le s$, then by (\ref{SdiM2}),
\begin{equation}\label{Jin2}
\di_{1,m_s}K=\di_{m_p+1,m_s}(\di_{1,m_p}K)\supset S_{m_p+1,m_s}(\di_{1,m_p}K).
\end{equation}
As $s\to\infty$, the body on the left converges to $J$, while because $(H_m)$ is Schwarz-universal, the body on the right converges to the ball $B_{p,K}$ with center at the origin such that $V_n(B_{p,K})=V_n(\di_{1,m_p}K)$.  However, the latter equation implies that $V_n(B_{p,K})\to a$ as $p\to\infty$.  Now $V_n$ is strictly monotonic on ${\mathcal{K}}^n_n$, $J\supset B_{p,K}$ by (\ref{Jin2}), and $V_n(J)=a$.  These facts force $J$ to be the ball $B_1$ centered at the origin with $V_n(B_1)=a$.  Consequently, any convergent subsequence of $(\di_{1,m}K)$ converges to $B_1$ and hence $\di_{1,m}K\to B_1$ as $m\to \infty$.

Finally, if $l\in \N$, $l\ge 2$, we can apply the above argument to the Schwarz-universal sequence $(H_{m+l-1})$, $m\in \N$, to conclude that $\di_{l,m}K$ converges to a ball $B_l$ as $m\to \infty$. This proves that $(H_m)$ is weakly $\di$-universal.
\end{proof}

Note that by \cite[Theorem~7.5]{BGG}, the right-hand inclusion in (\ref{SdiM2}) is satisfied if $\di_H$ is strictly monotonic, idempotent, invariant on $H$-symmetric cylinders, and invariant under translations orthogonal to $H$ of $H$-symmetric sets, and examples given after that result show that none of these four conditions can be dropped.  Also, the symmetrization from \cite[Example~10.7]{BGG} has all four properties but fails the left-hand inclusion in (\ref{SdiM2}).  The latter symmetrization is not rotationally symmetric, so one might hope that (\ref{SdiM2}) would hold if this fifth condition is added to the four others.  If this were true, we would have a corollary to Theorem~\ref{2} analogous to Proposition~\ref{Successivecor}.  (In this connection, note that $S_HK\subset I_HK$ is not true in general, yet $I_H$ has all five properties except that it is monotonic but not strictly monotonic.)  However, it is false.  To see this, in Example~\ref{ex22}, let $f(K)$ be the reciprocal of the first eigenvalue $\lambda_1(K)$ of the Laplace operator.  See \cite[Section~2.1]{Col05} for the definition of $\lambda_1(K)$ and the fact that it is homogeneous of degree $-2$ and satisfies a Brunn-Minkowski inequality with exponent $-1/2$, a result due to Brascamp and Lieb \cite{BL76}.  Since $f(K)=\lambda_1(K)^{-1}$, $f$ is homogeneous of degree $2$ and satisfies (\ref{BrunnM}) with $k=1/2$.  Moreover, $\lambda_1(K)$, and therefore $f(K)$, is rigid-motion invariant, and $f(K)$ is strictly increasing on ${\mathcal{K}}^n_n$ since $\lambda_1(K)$ is strictly decreasing; see \cite[p.~13]{Hen06}.  Finally, the Faber-Krahn inequality (see, for example, \cite[Theorem~3.2.1]{Hen06}) can be expressed in the form
$$
\left(\frac{f(K)}{f(B^n)}\right)^{1/2}=
\left(\frac{\lambda_1(K)}{\lambda_1(B^n)}\right)^{-1/2}\le \left(\frac{V_n(K)}{\kappa_n}\right)^{1/n},
$$
with equality if and only if $K$ is a ball.  From this (applied with $n$ replaced by $n-i$) it is easy to check that if $\di K$ is the symmetral from Example~\ref{ex22} with $f(K)=\lambda_1(K)^{-1}$, then $\di K\subset S_HK$, where the containment is strict in general and hence the left-hand inclusion in (\ref{SdiM2}) is false.  An example with similar properties can be obtained by instead taking $f(K)=\tau(K)$, the torsional rigidity of $K$; see \cite[Section~2.3]{Col05}.

\section{Symmetrals of compact sets}\label{compact}

In this section, we consider symmetrals of compact sets, paying special attention to Steiner, Schwarz, and Minkowski symmetrization.  The definitions of $S_HK$ and $M_HK$ given for $K\in {\mathcal{K}}^n$ in Section~\ref{symm} apply equally to $K\in {\mathcal{C}}^n$.

It is known (see \cite[Examples~2.1 and~2.4]{BBGV}) that in general the limit of a sequence of successive Steiner symmetrals of a compact convex set may not exist.  Also, the limit of a sequence of successive Steiner symmetrals of a compact set may exist but be non-convex; this is shown by \cite[Example~2.1]{BBGV} with the sequence ($\alpha_m$) of reals chosen so that their sum converges.

In \cite[Theorem~6.1]{BBGV}, it is proved that Klain's theorem \cite[Theorem~5.1]{Kla12} holds for compact sets, i.e., the same statement holds when the initial set is an arbitrary compact set.  Vol\v{c}i\v{c} \cite{V} showed that any sequence $(u_m)$ dense in $S^{n-1}$ can be ordered so that the resulting sequence is Steiner-universal, even if the initial set is an arbitrary compact set.

Despite this progress, many problems concerning symmetrals of compact sets remain open, even for Steiner symmetrization; see, for example, \cite[p.~1708]{BBGV}, \cite[p.~1691]{V}, and Section~\ref{problems}.

Our first contribution is the following extension of the first part of Theorem~\ref{thm5cor}.  (The second part would require a suitable extension of the definition of Minkowski-Blaschke symmetrization, which we shall not pursue here.)

\begin{thm}\label{july24thm}
Klain's theorem \cite[Theorem~5.1]{Kla12} (cf.~Theorem~\ref{4}) holds for Schwarz symmetrization of compact sets, where the limit set $L$ is rotationally symmetric with respect to each $U_j\in {\mathcal{F}}$ that appears infinitely often in $(H_m)$.
\end{thm}

\begin{proof}
We shall only give a sketch, indicating the necessary observations that allow the proof for Steiner symmetrization from \cite[Theorem~6.1]{BBGV} to be modified.

Let $i\in \{1,\dots,n-2\}$ and let $H\in {\mathcal{G}}(n,i)$. Let $K, L\in {\mathcal{C}}^n$ be nonempty and let $E_{\delta}=E+\delta B^n$ for $E\in {\mathcal{C}}^n$ and $\delta>0$.  The main argument requires the following four preliminary observations.  Firstly, the monotonicity of $S_H$ implies that $S_H(K\cap L)\subset S_HK\cap S_HL$.  This, the equality ${\mathcal{H}}^n(S_HK\setminus S_HL)={\mathcal{H}}^n(S_HK)-{\mathcal{H}}^n(S_HK\cap S_HL)$, and the fact that $S_H$ preserves volume yield
\begin{equation}\label{july241}
{\mathcal{H}}^n(S_HK\setminus S_HL)\le {\mathcal{H}}^n(K\setminus L).
\end{equation}
Secondly, it follows that ${\mathcal{H}}^n(S_HK\triangle S_HL)\le {\mathcal{H}}^n(K\triangle L)$ and hence that $S_H$ is continuous on ${\mathcal{C}}^n$ in the symmetric difference metric.  Thirdly, the inclusion
\begin{equation}\label{july242}
(S_HK)_{\delta}\subset S_H K_{\delta}
\end{equation}
is the special case $L=\delta B^n$ of $S_HK+S_HL\subset S_H(K+L)$, which in turn follows from the Brunn-Minkowski inequality applied to intersections of $K$ and $L$ with translates of $H^{\perp}$.  Fourthly,
\begin{equation}\label{july243}
{\mathcal{H}}^n(S_HK_{\delta}\setminus rB^n)= {\mathcal{H}}^n(K_{\delta}\setminus rB^n)~{\text{for all}}~\delta, r>0\Rightarrow S_HK=K.
\end{equation}
To see this, let $u\in S^{n-1}\cap H^{\perp}$.  Then
$${\mathcal{H}}^n(K_{\delta}\setminus rB^n)\ge {\mathcal{H}}^n(S_{u^{\perp}}K_{\delta}\setminus rB^n)\ge {\mathcal{H}}^n(S_HS_{u^{\perp}}K_{\delta}\setminus rB^n) =
{\mathcal{H}}^n(S_HK_{\delta}\setminus rB^n),$$
where the inequalities follow from (\ref{july241}) and the equality from $S_HS_{u^{\perp}}K_{\delta}=S_HK_{\delta}$.  The hypothesis in (\ref{july243}) implies that equality holds throughout. Then, from ${\mathcal{H}}^n(K_{\delta}\setminus rB^n)= {\mathcal{H}}^n(S_{u^{\perp}}K_{\delta}\setminus rB^n)$ and \cite[Lemma~3.3]{BBGV}, we get $S_{u^{\perp}}K=K$ and since $u\in S^{n-1}\cap H^{\perp}$ was arbitrary, the desired conclusion $S_HK=K$ in (\ref{july243}) follows.

With these preliminary observations in hand, a few substitutions allow the main argument of \cite[Theorem~6.1]{BBGV} to be followed without difficulty.  Of course Steiner symmetrization with respect to a sequence of directions must be replaced by Schwarz symmetrization with respect to a sequence of subspaces. Otherwise, it is only necessary to appeal to (\ref{july241}), (\ref{july242}), and (\ref{july243}) wherever the proof of \cite[Theorem~6.1]{BBGV} uses \cite[(2.1)]{BBGV}, \cite[(2.2)]{BBGV}, and \cite[Lemma~3.3]{BBGV}, respectively.
\end{proof}

For the remaining results in this section, we need some definitions. The {\em Kuratowski limit superior} of a sequence $(A_m)$ of sets in $\R^n$ is the closed set
$$\Ls_{m\to\infty} A_m=\{x\in\R^n: \forall \ee>0,~B(x,\ee)\cap A_m\neq\emptyset~{\text{for infinitely many }}m\}$$
and the {\em Kuratowski limit inferior} of $(A_m)$ is the closed set
$$\Li_{m\to\infty} A_m=\{x\in\R^n: \forall \ee>0,~B(x,\ee)\cap A_m\neq\emptyset~{\text{for sufficiently large }}m\}.$$
It is not hard to check that $\Li_{m\to\infty} A_m\subset \Ls_{m\to\infty} A_m$.  Also, if $(A_m)$ is uniformly bounded, then $\lim_{m\to\infty} A_m$ exists if and only if $\Li_{m\to\infty} A_m= \Ls_{m\to\infty} A_m$.  See \cite[Section~29]{K}, where these notions are credited to P.~Painlev\'{e}.

\begin{lem}\label{lemdec12}
Let $i\in \{1,\dots,n-1\}$, let $\di$ be an $i$-symmetrization process on ${\mathcal{C}}^n$ that preserves sets in ${\mathcal{K}}^n_n$, and let $(H_m)$ be a $\di$-universal sequence in ${\mathcal{G}}(n,i)$.   Let $K\in {\mathcal{C}}^n$ and suppose that the sequence $(\di_{1,m}K)$, defined via \eqref{dec121}, is uniformly bounded.  Let $r>0$ be minimal such that $\Ls_{m\to\infty}\di_{1,m}K\subset rB^n$. Let $f:{\mathcal{K}}^n_n\to [0,\infty)$ be continuous and strictly increasing. If $\di_H$ is monotonic on ${\mathcal{C}}^n$ and $f$-preserving on ${\mathcal{K}}^n_n$ for each $H\in {\mathcal{G}}(n,i)$, then $rS^{n-1}\subset \Li_{m\to\infty}\di_{1,m}K$.
\end{lem}

\begin{proof}
Let $K\in {\mathcal{C}}^n$ and $r>0$ satisfy the assumptions of the lemma and let $\ee>0$.  Let $L=\Ls_{m\to\infty}\di_{1,m}K$ and let $I=\Li_{m\to\infty}\di_{1,m}K$.  We claim that there is an $m_0\in \N$ such that
\begin{equation}\label{jan61}
\di_{1,m}K\subset L+\ee B^n
\end{equation}
for $m\ge m_0$.  Indeed, if this is not true, then there exists a subsequence $(m_j)$ and points
$$x_{m_j}\in (\di_{1,m_j}K)\setminus(L+\ee B^n)$$
for $j\in \N$.  Since $(\di_{1,m}K)$ is uniformly bounded, there is a subsequence of $(x_{m_j})$ converging to some $z$.  But then $z\in L$ and $z\not\in L+\ee B^n$, which is impossible.

If the conclusion of the lemma is false, there is an $x\in (rS^{n-1})\setminus I$ and hence an $\ee_0>0$ and subsequence $(m_k)$ such that
\begin{equation}\label{jan62}
B(x,\ee_0)\cap \di_{1,m_k}K=\emptyset
\end{equation}
for $k\in \N$.  Let $J$ be a closed half-space such that $x\not\in J$ and
\begin{equation}\label{jan63}
(rB^n)\setminus B(x,\ee_0)\subset J.
\end{equation}
Let $C=(rB^n)\cap J$. As $f$ is continuous and strictly increasing, we can choose $t>0$ small enough so that $x\not\in C+tB^n$ and
\begin{equation}\label{jan64}
f(C+tB^n)< f(rB^n).
\end{equation}
By (\ref{jan63}), $(rB^n)\setminus B(x,\ee_0)\subset C\setminus B(x,\ee_0)$ and hence
\begin{equation}\label{jan65}
((r+t')B^n)\setminus B(x,\ee_0)\subset (C+tB^n)\setminus B(x,\ee_0)
\end{equation}
for sufficiently small $t'>0$.  It then follows from (\ref{jan61}) with $\ee=t'$, the assumption $L\subset rB^n$, (\ref{jan62}), and (\ref{jan65}) that
\begin{equation}\label{jan66}
\di_{1,m_k}K\subset C+tB^n
\end{equation}
for $m_k\ge m_0$.  Now if $p>m_k\ge m_0$, (\ref{jan66}) and the monotonicity of $\di_H$ on ${\mathcal{C}}^n$ yields
\begin{equation}\label{jan67}
\di_{1,p}K=\di_{m_k+1,p}(\di_{1,m_k}K)\subset \di_{m_k+1,p}(C+tB^n) \to sB^n
\end{equation}
as $p\to\infty$, where $s$ is independent of $k$, since $C+t B^n\in{\mathcal{K}}^n_n$ and $(H_m)$ is $\di$-universal.  As $\di_H$ is $f$-preserving on ${\mathcal{K}}^n_n$ and $f$ is continuous there, we deduce from (\ref{jan64}) and (\ref{jan67}) that
$$f(s B^n)=f(\di_{{m_k+1,p}}(C+t B^n))=f(C+t B^n)<f(rB^n)$$
for $p\ge m_k\ge m_0$, so $s<r$ because $f$ is strictly increasing on ${\mathcal{K}}^n_n$.  This and (\ref{jan67}) imply that $\di_{1,p}K\subset ((r+s)/2)B^n$ for sufficiently large $p$ and consequently $L\subset ((r+s)/2)B^n$.  Since $(r+s)/2<r$, this contradicts the minimality of $r$.
\end{proof}

\begin{thm}\label{thmdec17}
Let $i\in \{1,\dots,n-1\}$ and let $(H_m)$ be a Schwarz-universal sequence in ${\mathcal{G}}(n,i)$.   If $K\in {\mathcal{C}}^n$ and $l\in \N$, the successive Schwarz symmetrals $S_{l,m}K$ of $K$, defined by \eqref{dec121} with $\di=S$, converge to $rB^n$ as $m\to\infty$, for some $r=r(K)$ independent of $l$.  In other words, $(H_m)$ is also Schwarz-universal for compact sets.
\end{thm}

\begin{proof}
It is enough to consider the case when $l=1$, because the fact that Schwarz symmetrization preserves volume on ${\mathcal{C}}^n$ means that the constant $r$ in the statement of the theorem must satisfy $V_n(rB^n)={\mathcal{H}}^n(K)$.  Accordingly, let $L=\Ls_{m\to\infty}S_{1,m}K$ and let $I=\Li_{m\to\infty}S_{1,m}K$.  Let $r\ge 0$ be minimal such that $L\subset rB^n$.  If $r=0$, then $L=\{o\}$ and by (\ref{jan61}), the result holds with $r=0$.  Otherwise, $r>0$ and we may apply Lemma~\ref{lemdec12} with $f=V_n$ to conclude that $rS^{n-1}\subset I$.  It will suffice to prove that $I=rB^n$, since this implies that $I=L$.

Suppose that $I\neq rB^n$.  Then there is an $x_0\in rB^n\setminus I$ and hence an $\ee_0>0$ and a subsequence $(m_k)$ such that
\begin{equation}\label{jan71}
B(x_0,\ee)\cap S_{1,m_k}K=\emptyset
\end{equation}
for $0<\ee\le \ee_0$ and $k\in \N$.  As $I\subset rB^n$ is compact, there is an $0<\ee\le \ee_0$ such that
\begin{equation}\label{dec171}
B(x_0,\ee)\subset rB^n\setminus (I+\ee B^n)\subset (r-\ee)B^n.
\end{equation}
Choose $v_1,\dots,v_p\in rS^{n-1}$ such that
$$rS^{n-1}\subset\cup\{B(v_i,\ee/2): i=1,\dots,p\}.$$
Let $i\in\{1,\dots,p\}$.  Since $v_i\in I$, there is an $n_i$ such that for each $m\ge n_i$, there is an $x_i\in B(v_i,\ee/2)\cap S_{1,m}K$.  Let $n_0=\max\{n_i: i=1,\dots,p\}$ and choose $k\in \N$ with $m_k\ge n_0$.

Let $x_0=y_0+z_0$, where $y_0\in H_{m_k}$ and $z_0\in H_{m_k}^{\perp}$.  Let $v_0=y_0+w_0\in rS^{n-1}$, where $w_0\in H_{m_k}^{\perp}$ and $x_0\in [y_0,v_0]$.  Since $v_0\not\in B(x_0,\ee)$ by (\ref{dec171}), we have $\|x_0-v_0\|> \ee$.  Moreover, $z_0\in [o,w_0]$ and $\|z_0-w_0\|=\|x_0-v_0\|$.  Consequently,
\begin{equation}\label{dec181}
\|w_0\|> \|z_0\|+\ee.
\end{equation}
Choose $i\in \{1,\dots,p\}$ such that $v_0\in B(v_i,\ee/2)$. Since ${m_k}\ge n_0$, there is an $x_i\in B(v_i,\ee/2)\cap S_{1,{m_k}}K$ and hence $\|v_0-x_i\|\le\ee$.  Let $x_i=y_i+w_i$, where $y_i\in H_{m_k}$ and $w_i\in H_{m_k}^{\perp}$.  Then $\|y_0-y_i\|\le\ee$ and $\|w_0-w_i\|\le\ee$, and from the latter and (\ref{dec181}), we conclude that $\|w_i\|> \|z_0\|$.  It follows that $$(B(y_i,\|w_i\|)\cap (H_{m_k}^{\perp}+y_i))\cap B(x_0,\ee)\neq\emptyset.$$
However, $B(y_i,\|w_i\|)\cap (H_{m_k}^{\perp}+y_i)\subset S_{1,{m_k}}K$, because $x_i=y_i+w_i\in S_{1,{m_k}}K$ and $S_{1,{m_k}}K$ is a Schwarz symmetral with respect to $H_{m_k}$.  This contradicts (\ref{jan71}) and completes the proof.
\end{proof}

\begin{thm}\label{thmjan9}
Let $i\in \{1,\dots,n-1\}$ and let $(H_m)$ be a Minkowski-universal sequence in ${\mathcal{G}}(n,i)$.   If $K\in {\mathcal{C}}^n$ and $l\in \N$, the successive Minkowski symmetrals $M_{l,m}K$ of $K$, defined by \eqref{dec121} with $\di=M$, converge to $rB^n$ as $m\to\infty$, for some $r=r(K)$ independent of $l$.  In other words, $(H_m)$ is also Minkowski-universal for compact sets.
\end{thm}

\begin{proof}
It is enough to consider the case when $l=1$.   Indeed, since $\conv(A+B)= \conv A + \conv B$ for arbitrary sets $A$ and $B$ in $\R^n$, Minkowski symmetrization preserves the mean width of convex hulls.  Then, for any $l\in\N$, if $M_{l,m}K$ converges to $rB^n$, we have that $M_{l,m}(\conv K)$ also converges to $rB^n$. But $r$ must satisfy $V_1(rB^n) = V_1(\conv K)$, so $r$ is independent of $l$.

Let $L=\Ls_{m\to\infty}M_{1,m}K$ and let $I=\Li_{m\to\infty}M_{1,m}K$.  Let $r\ge 0$ be minimal such that $L\subset rB^n$.  If $r=0$, then $L=\{o\}$ and by (\ref{jan61}), the result holds with $r=0$.  Otherwise, $r>0$ and we may apply Lemma~\ref{lemdec12} with $f=V_1$ to conclude that $rS^{n-1}\subset I$.  It will suffice to prove that $I=rB^n$, since this implies that $I=L$.

Suppose that $I\neq rB^n$.  Then there is an $x\in rB^n\setminus I$ and hence an $\ee>0$ and a subsequence $(m_k)$ such that
\begin{equation}\label{jan91}
B(x,\ee)\cap M_{1,m_k}K=\emptyset
\end{equation}
for $k\in \N$.  Choose $v_1,\dots,v_p\in rS^{n-1}$ such that
$$rS^{n-1}\subset\cup\{B(v_i,\ee/2): i=1,\dots,p\}.$$
Let $i\in\{1,\dots,p\}$.  Since $v_i\in I$, there is an $n_i$ such that for each $m\ge n_i$, there is an $x_i\in B(v_i,\ee/2)\cap M_{1,m}K$.  Let $n_0=\max\{n_i: i=1,\dots,p\}$ and choose $k\in \N$ with $m_k> n_0$.

Let $y, z\in rS^{n-1}$ be such that $x=(y+z)/2$.
Choose $i,j\in \{1,\dots,p\}$ such that $y\in B(v_i,\ee/2)$ and $R_{H_{m_k}}z\in B(v_j,\ee/2)$. Since ${m_k}> n_0$, there are $x_i\in B(v_i,\ee/2)\cap M_{1,{m_k}-1}K$ and $x_j\in B(v_j,\ee/2)\cap M_{1,{m_k}-1}K$.  Then $x_i\in B(y,\ee)$ and $x_j\in B(R_{H_{m_k}}z,\ee)$; the latter implies that $R_{H_{m_k}}x_j\in B(z,\ee)$.

Let $q=(x_i+R_{H_{m_k}}x_j)/2$.  Then $q\in B(x,\ee)$ and $q\in M_{H_{m_k}}(M_{1,{m_k}-1}K)=M_{1,{m_k}}K$
since $x_i\in M_{1,{m_k}-1}K$ and $R_{H_{m_k}}x_j\in R_{H_{m_k}}(M_{1,{m_k}-1}K)$. It follows that $B(x,\ee)\cap M_{1,{m_k}}K\neq\emptyset$, which contradicts (\ref{jan91}) and completes the proof.
\end{proof}

Theorem~\ref{thmdec17} with $i=n-1$ and Theorem~\ref{thmjan9}, together with the results mentioned at the end of Section~\ref{symm}, show that the eight properties (weakly) Steiner-universal, (weakly) Minkowski-universal, (weakly) Steiner-universal for compact sets, and (weakly) Minkowski-universal for compact sets are all equivalent.

\section{Open problems}\label{problems}

\begin{prob}\label{prob0}
Can the assumptions in Theorem~\ref{corthm4} and Proposition~\ref{Successivecor} be weakened?
\end{prob}

We do not know if the assumptions of continuity and invariance under translations orthogonal to $H$ of $H$-symmetric sets are needed for Theorem~\ref{corthm4}, nor whether the latter condition is needed for Proposition~\ref{Successivecor}.

Regarding the continuity property, note that $S_H$ is not continuous on ${\mathcal{K}}^n$, despite having all the properties considered in \cite{BGG} except projection covariance.  Of course $S_H$ is continuous on ${\mathcal{K}}^n_n$, but Example~\ref{44exd} exhibits a $\di_H$ that is monotonic, invariant on $H$-symmetric sets, and invariant under translations of $H$-symmetric sets orthogonal to $H$, but not continuous.  It may be that projection covariance, either alone or in combination with some other properties, implies continuity.  Certainly this is the case when $i=0$, as was proved in \cite[Corollary~8.3]{GHW}.  The latter relied on \cite[Theorem~8.2]{GHW}, while for $i\in\{1,\dots,n-2\}$ we have only the weaker \cite[Proposition~4.4]{BGG} and for $i=n-1$ nothing at all.

\begin{prob}\label{prob7}
Let $i\in \{1,\dots,n-1\}$, let $\di$ be an $i$-symmetrization process on ${\mathcal{C}}^n$ that satisfies the assumptions of Lemma~\ref{lemdec12}, and let $(H_m)$ be a $\di$-universal sequence in ${\mathcal{G}}(n,i)$ for convex bodies. Is $(H_m)$ also $\di$-universal for compact sets?
\end{prob}

\begin{prob}\label{prob10}
Do Theorem~\ref{4} and the second statement in Theorem~\ref{corthm4}, i.e., Klain's theorem for fiber and Minkowski symmetrization, hold if the initial set is an arbitrary compact set?
\end{prob}

As we remarked at the beginning of Section~\ref{compact}, the answer is positive for the case $i=n-1$ of Theorem~\ref{4}, corresponding to Steiner symmetrization, by \cite[Theorem~6.1]{BBGV}.

\begin{prob}\label{probmar28}
Does Klain's theorem hold for Blaschke symmetrization (see \cite[p.~59]{BGG})?
\end{prob}

\end{document}